\documentclass[10pt,leqno,a4paper]{amsart}
\usepackage[OT2,T1]{fontenc}
\usepackage{amsfonts, amssymb, amsmath, amsthm, latexsym, enumerate, epsfig, color, mathrsfs, fancyhdr, pdfsync, eurosym, endnotes, stmaryrd, bm, bbm}
\usepackage[all]{xy}
\usepackage{graphicx}
\usepackage{epstopdf}
\usepackage{setspace}
\usepackage{pdfsync}
\definecolor{labelkey}{rgb}{0,1,0}
\usepackage{caption}
\usepackage[hypertex]{hyperref}
\usepackage{amsmath,amscd,amssymb}
\usepackage[english]{babel}
\usepackage[inner=3cm,outer=3cm,bottom=5cm]{geometry}
\usepackage{setspace}
\usepackage{mathrsfs}
\usepackage[all]{xy}
\usepackage{tikz-cd}
\usepackage{float}
\usepackage{newfloat}

\DeclareFloatingEnvironment{diagram}
\addto\captionsczech{%
}

\newtheorem{thm}{Theorem}[section]
\newtheorem{cor}[thm]{Corollary}
\newtheorem{lemma}[thm]{Lemma}

\newtheorem{prop}[thm]{Proposition}

\newtheorem{defn}[thm]{Definition}

\theoremstyle{remark}

\theoremstyle{definition}

\newtheorem{rmk}[thm]{Remark}

\numberwithin{equation}{thm}

\def\beq{\begin{equation}}
\def\eeq{\end{equation}}

\def\ben{\begin{enumerate}}
\def\een{\end{enumerate}}

\usepackage[OT2,T1]{fontenc}
\DeclareSymbolFont{cyrletters}{OT2}{wncyr10}{m}{n}
\SetSymbolFont{cyrletters}{bold}{OT2}{cmr}{bx}{n}
\DeclareMathSymbol{\Dc}{\mathalpha}{cyrletters}{68}

\def\crash#1{}

\def\P{{\mathbb P}}
\def\Q{{\mathbb Q}}
\def\R{{\mathbb R}}
\def\C{{\mathbb C}}

\def\bU{{\mathbb U}}

\def\l{\left}
\def\r{\right}
\def\[[{\l[\l[}
\def\]]{\r]\r]}

\def\lc{\emph{loc.cit.}\;}

\def\cB{{\mathcal B}}

\def\cO{{\mathcal O}}

\def\cL{{\mathcal L}}

\def\cP{{\mathcal P}}
\def\cR{{\mathcal R}}
\def\cS{{\mathcal S}}

\def\cV{{\mathcal V}}
\def\cU{{\mathcal U}}

\def\sD{{\mathscr D}}

\def\sP{{\mathscr P}}

\def\sU{{\mathscr U}}

\def\T{{\mathbb T}}

\def\fr{{\mathfrak r}}
\def\wtilde{\widetilde}

\def\a{\alpha}

\def\id{{\rm id\,}}
\def\ker{{\rm Ker\,}}

\def\iso{\xrightarrow{\ \sim\ }}

\def\vphi{\varphi}
\def\eps{\epsilon}

\def\vu{{\vec{u}}}
\def\vy{{\vec{y}}}

\def\ue{{\underline{e}}}

\def\uV{{\underline{V}}}

\def\bU{{\mathbf{U}}}
\def\bV{{\mathbf{V}}}

\begin{document}
\setlength{\baselineskip}{0.55cm}	
\title[Optimal bases]{Optimal bases for direct images of $p$-adic differential modules over discs}
\author{Velibor Bojkovi\'c}
\address{Laboratoire de Math\'ematiques Nicolas Oresme, Caen, France}
\email{velibor.bojkovic@unicaen.fr}


\begin{abstract}
Let $(k,|\cdot|)$ be a complete and algebraically closed valued field extension of $(\mathbb{Q}_p,|\cdot|_p)$. Given a finite morphism $\vphi:\sD_1\to \sD_2$ of unit discs over $k$, a differential module $(M,D)$ on $\sD_1$ and a point $b\in \sD_2(k)$, we construct explicitly an optimal basis of space of horizontal elements for the direct image $\vphi_*(M,D)$ at $b$ in terms of the suitable chosen optimal bases of $(M,D)$ at preimages of $b$ by $\vphi$ and ramification properties of the morphism.
\end{abstract}
\maketitle
\tableofcontents

\section{Introduction}

It is a standard fact that given a linear system of differential equations with coefficients in $\C(t)$, that every solution at a regular point $t_0$ has radius of convergence at least equal to the distance of $t_0$ to the nearest singularity.

This fact utterly fails when one considers such systems with coefficients in $k(t)$, where $k$ is a non-archimedean field. For example, if $(k,|\cdot|)$ is a complete valued field extension of $(\Q_p,|\cdot |_p)$, where $|\cdot|_p$ is $p$-adic norm corresponding to the prime number $p$, then nontrivial solutions of $y'(t)=y(t)$ converge around zero only on the open disc of radius $|p|^{\frac{1}{p-1}}$.

However, one may argue that it is precisely this anomaly that inspired much of the rich theory of $p$-adic differential equations (modules) ever since Dwork's proof of rationality of the Weil zeta function back in 60's. \\

The problems that we study in this article, generally speaking, are related to finding ``optimal'' solutions of $p$-adic differential systems, that is, the solutions having the best possible radii of convergence at a chosen point. 

To state the problems more precisely, let $(k,|\cdot|)$ be as before and in addition assume that $k$ is algebraically closed, and let $\cO_t(a,r^-)$, where $r\in \R_{>0}$, denote the ring of those elements in $k[[t-a]]$ which converge for every $t$ with $|t-a|<r$. By a ($p$-adic) differential module of rank $\fr$ we mean a free $\cO_t(0,1^-)$-module $M$ of rank $\fr$ equipped with a $k$-linear map $D_t:M\to M$ satisfying the Leibnitz rule $D_t(f(t)\cdot m)=f'(t)\cdot m+ f(t)\cdot D_t(m)$. For example, one may take the trivial differential module $(T,D_t)$, which is defined with $T=\cO_t(0,1^-)\cdot e$ and $D_t$ is given by $D_t(e)=0\cdot e$. The elements of the $K$-vector space $\ker(D_t)$ are called horizontal. Once we choose a basis $\{e_1,\dots,e_\fr\}$ for $M$, and construct an $\fr\times \fr$ matrix $A(t):=-(a_{i,j}(t))^T$ where $D_t(e_i)=\sum_{j=1}^\fr a_{i,j}(t)\cdot e_i$, then we see that the horizontal elements correspond precisely to the solutions of the differential system 
\begin{equation}\label{eq: intro}
\frac{d}{dt}\vec{Y}(t)=A(t)\,\vec{Y}(t), 
\end{equation}
that belong to $\cO_t(0,1^-)^\fr$. For example, in case of $(T,D_t)$, the previous system becomes simply $\frac{d}{dt}y(t)=0\cdot y(t)$, hence its space of solutions is $k$, that is, the horizontal elements of $T$ are of the form $\alpha\cdot e$ for some $\alpha\in k$.  However, for every point $a\in k$, $|a|<1$ (we denote this set by $\sD_t^-(k)$) one has a space of solutions of system \eqref{eq: intro} that for the purpose of this introduction we denote by $\cS_a=\cS_a(M,D_t)\subset k[[t-a]]^\fr$. Then, $p$-adic Cauchy theorem assures that $\dim_k\cS_a=\fr$, which when translated to the language of differential modules means that after the extension of base ring $\cO_t(0,1^-)\to \cO_t(a,r^-)$, $\dim_k\ker_k(D_t)=\fr$. 

If now $\vec{Y}(t)\in\cS_a$, then the number $\cR_a\big(\vec{Y}(t)\big)\:=\sup\{r\in (0,1)\mid \vec{Y}(t)\in\cO_t(a,r^-)^\fr\}$ is called the radius of convergence of $\vec{Y}(t)$ (or of the corresponding element $m\in M$ as this number does not depend on the chosen basis) at $a$.

One further can consider the following numbers
$$
\cR_{a,i}:=\sup\{r\in (0,1)\mid \dim_k(\cS\cap \cO_t(a,r^-))\geq \fr-i+1\}.
$$
The $\fr$-tuple of numbers $\cR_a=\cR_a\big((M,D_t)\big):=(\cR_{a,1},\dots,\cR_{a,\fr})$ is called the multiradius of convergence of solutions of \eqref{eq: intro} at $a$. It is an important invariant of $(M,D_t)$, that is, it does not depend on the chosen basis. Then, finally we say that a basis $\vec{Y}_{a,1}(t),\dots,\vec{Y}_{a,\fr}(t)$ of $\cS_a$ is optimal if the radii of convergence $\cR_a\big(\vec{Y}_{a,1}(t)\big),\dots,\cR_a\big(\vec{Y}_{a,\fr}(t)\big)$ can be rearranged to form $\cR_a$. 

The multiradius of convergence $\cR_a$ has been extensively studied, most notably in \cite{Bal10}, \cite{Pul}, \cite{Poi-Pul2}, \cite{kedlayacompositio} just to name a few references, in the context of Berkovich curves and modules with connections on them, and where many properties such as continuity and harmonicity have been established.  For a systematic approach one may consult \cite{Kedbook} and references therein. 

A standard way, vaguely speaking, to establish some general property of $p$-adic differential modules (such as those already mentioned, continuity and harmonicity for example) is to first establish the property for differential modules which have small (components of) multiradius for every $a$ considered, or for differential modules which have each $\cR_{a,i}=1$. Then, the general case is established by reducing it to these particular ones by using some more or less familiar constructions on differential modules and being careful that we can control the behavior of the desired property under the applied constructions.

One of these constructions, that has proved to be quite useful is that of a direct image of a differential module by a finite \'etale morphism. If one considers a finite \'etale morphism $\vphi:\sD_t^-\to \sD_s^-$ of degree $d$ between open unit discs equipped with coordinates $t$ and $s$, respectively, then $\vphi$ can be represented as $s=f(t)$, where $f(t)\in\cO_t(0,1^-)$. Then, from $(M,D_t)$ one may obtain differential module $(M_\vphi,D_s)$ over $\cO_s(0,1^-)$ by simply taking $M_\vphi$ to be $M$ seen as a module over $\cO_s(0,1^-)$ via $\vphi$ and define the action of $D_s$ by $D_s\big((m)_\vphi\big):=\big(\frac{1}{f'(t)}D_t(m)\big)_\vphi$ (where $\vphi$ in the index means that we see the respective element of $M$ as an element of $M_\vphi$). Such established  differential module $(M_\vphi,D_s)$ will have rank $\fr\cdot d$ and is called direct image of $(M,D_t)$ by $\vphi$. 

If we take $e_1,\dots,t^{d-1}\cdot e_1,\dots,t^{d-1}\cdot e_\fr$ for a basis of $M_\vphi$ over $\cO_s(0,1^-)$, we obtain the associated linear system of differential equations

\begin{equation}\label{eq: intro 2}
 \frac{d}{ds}\vec{Z}(s)=B(s)\,\vec{Z}(s),
\end{equation}
where $B(s)\in M_{\fr\cdot d}(\cO_s(0,1^-))$. Let now $b\in \sD_s^-(k)$ and $a_1,\dots,a_d\in\sD_t^-(k)$ be all the preimages of $b$ by $\vphi$. The general questions that we deal with in the present articles are: How can one obtain $\cS_b(M_\vphi,D_s)$ from $\cS_{a_i}(M,D)$, $i=1,\dots,d$? How can one obtain an optimal basis of $\cS_b(M_\vphi,D_s)$ starting from some optimal bases of $\cS_{a_i}(M,D_t)$, $i=1,\dots,d$?\\

The motivation for these questions comes from \cite{BojPoi}, where authors have obtained a precise formula for $\cR_b\big((M_\vphi,D_s)\big)$ in terms of $\cR_{a_i}\big((M,D_t)\big)$ and geometric properties of the morphism $\vphi$. The formula has then been applied in \lc to study the ramification properties of certain extensions of valued fields that appear in the context of Berkovich curves as well as in establishing general ``laplacian'' formula for the multiradius of convergence of differential modules on Berkovich curves. Furthermore, it has been used in \cite{Ba-Bo} to study metric properties and skeleta of finite morphisms of Berkovich analytic curves.
 
The present article aims to refine and complete the main results of \cite{BojPoi} by answering the previous questions.\\

The ingredients that we use in our discussion are of algebro-geometric flavor, coming from the algebraic and ramification properties of $\vphi$. On the algebraic side, since morphism $\vphi$ is finite, it comes with a unique monic, degree $d$ polynomial $P(s,X)=a_0(s)+\dots+X^d\in \cO_s(0,1^-)[X]$ which has $t$ for a solution. One can also look for solutions in $k[[s-b]]$ and it turns out that for every $a_i$ as before, there is a solution $u_{a_i}(s)\in k[[s-b]]$ that has $a_i$ for the constant coefficient. Moreover, it has nonzero radius of convergence at $b$. Let $\bU(s)$ denote the Vandermonde $d\times d$ matrix $(v_{i,j})$ where $v_{i,j}=u_{a_i}(s)^{j-1}$, and let $\bV(s)$ be its inverse (which has entries with nonzero radius of convergence at $b$). Further, for each $i=1,\dots,d$, let $\mathbf{Y}_{a_i}(t)$ denote the $\fr\times \fr$ matrix whose columns are given by vectors $\vec{Y}_{a_i,1}(t),\dots,\vec{Y}_{a_i,\fr}(t)$. 

Then, we have the following (Theorem \ref{thm: fund sol} and Corollary \ref{cor: sol trivial})
\begin{thm}
 A basis of $\cS_b\big((M)_\vphi,D_s\big)$ is given by the columns of the matrix
 
 $$
 \Bigg(\bigoplus_{i=1}^\fr \bV(s)\Bigg)\Bigg(\bigoplus_{i=1}^d \mathbf{Y}_i\big(u_{a_i}(s)\big)\Bigg),
 $$ 
 where the direct sum is taken in the sense of matrices. 
 
 In particular, the solutions of the direct image of a trivial differential module are given by the columns of the matrix $\bV(s)$. 
\end{thm}

On the geometric side, we consider discs $\sD_t^-$ and $\sD_s^-$ to be in the sense of Berkovich $k$-analytic geometry and assume some familiarity with the structure and classification of points in the Berkovich projective line $\P^1_k$ as for example presented in \cite[Section 3.6.]{BerCoh} (although it is not indispensable, using the language of Berkovich geometry does quite simplify the discussion and it is a standard in the more recent literature on $p$-adic differential equations). For $a\in \sD_t^-$ and $r\in [0,1)$, let $\eta_{a,r}$ denote the maximal point of the closed (Berkovich) disc $\sD_t(a,r)$ with center at $a$ and of radius $r$. For each $a'\in \sD_t^-$ let $l_{\vphi,a'}$ denote the set of the points $\eta_{a',r}$, $r\in [0,1)$ and put $\T_{\vphi,b}:=\cup_{i=1}^dl_{\vphi,a_i}$. Then, $T_{\vphi,b}$ has a natural structure of a topological graph. We say that $\eta\in \T_{\vphi,b}$ is branching if $\T_{\vphi,b}\setminus\{\eta\}$ has more than $2$ connected components and we denote the set of branching points by $\cB_{\vphi,b}$. A branch $\eta\in \cB_{\vphi,b}$ is any open disc $\sD\subset \sD_t^-$ which has a nonempty intersection with $\T_{\vphi,b}$ and has $\eta$ for the maximal point. We put $\delta(\eta)$ to denote the number of branches at $\eta$.  

Now, for each  $\vec{Y}_{a_i,j}(t)$ as before, let $r_{i,j}$ be its radius of convergence at $a_i$. We say that $\big(\vec{Y}_{a_i,j},\sD_t(a_i,r_{i,j}^-)\big)$ is a fundamental pair and denote the set of different fundamental pairs by $\sP$ (it depends on the chosen basis at $a_i$). Let us fix a fundamental pair $\cP:=(\vec{Y}(t),\cU)$, and put $\cB_{\vphi,b}(\cP):=\cB_{\vphi,b}\cap \cU$. Further, for each $\eta\in \cB_{\vphi,b}(\cP)$ let us choose any $\delta(\eta)-1$ branches at $\eta$, and denote their set, together with $\cU$ by $\sU_\cP$. Finally, for $\cP=(\vec{Y}(t),\cV)\in \sP$ and $\cU\in \sU_\cP$, let  $v_{\cP,\cU,s}$ denote the column vector of dimension $\fr\cdot d$ of analytic functions in $k[[s-b]]$ defined in the following way: The $\fr$ entries of $v_{\cP,\cU,s}$ in positions $(i-1)\cdot \fr+1,\dots,i\cdot \fr$ are the $\fr$ entries of $\vec{Y}(u_{a_i}(s))$, if $a_i\in \cU$, and 0 otherwise, where $i=1,\dots,d$. Then, we have the following (Theorem \ref{thm: main}; see also Definition \ref{defn: linked} for the notion of linked bases)
\begin{thm}
 \begin{enumerate}
  \item The vector of analytic functions $\Big(\bigoplus_{i=1}^\fr\bV(s)\Big)\,v_{\cP,\cU,s}$ is a solution of the system \eqref{eq: intro 2}. Its radius of convergence is equal to the radius of the disc $\vphi(\cU)$. 
  
  \item Suppose that the basis $\{\vec{Y}_{a_i,1}(s),\dots,\vec{Y}_{a_i,\fr}(t)\}$ is optimal, $i=1,\dots,d$, and that they are linked. Then, an optimal basis of $(M_\vphi,D_s)$ at $b$ is given by $\Big(\bigoplus_{i=1}^\fr\bV(s)\Big)\,v_{\cP,\cU,s}$, $\cU\in\sU_\cP$, $\cP\in \sP$.
 \end{enumerate}
\end{thm}

\section{Finite morphisms of open discs}

\subsection{Notation}

\subsubsection{}
Throughout this article $k$ will denote an algebraically closed and complete valued extension of $\Q_p$. 
We further denote by:
\begin{itemize}
\item $\sD_t(a,r^-)$ (resp. $\sD_t(a,r)$) an open (resp. closed) (Berkovich) disc centered at $a\in k$ and of radius $r$ with respect to coordinate $t$. 

In general we will write $\sD_t(a,r^\pm)$ for a disc centered at $a$ and of radius $r$ when we do not want to specify if it is open or closed. Further, we will write $\sD_t$ (resp. $\sD_t^-$) for a closed (resp. open) unit disc.

\item $\cO_t(a,r^-)$ (resp. $\cO_t(a,r)$) the ring of analytic functions on disc $\sD_t(a,r^-)$ (resp. $\sD_t(a,r)$). These are functions $f(t)=\sum_{i \geq 0}a_i\cdot t^i\in k[[t]]$ such that 
$$
\lim\limits_{i\to\infty}|a_i|\cdot\rho^i\to 0,\quad \text{for all}\quad \rho\in [0,r)\quad \text{resp. } \rho\in [0,r].
$$
We will write $\cO_t^\pm$ for the ring of analytic functions on unit disc $\sD_t^\pm$.

\item Given $f(t)=\sum_{i\geq 0}a_i\cdot t^i\in\cO_t(r^-)$, we denote by $v(f(t),\cdot):(-\log r,\infty)\to \R$ the valuation polygon of $f$, which we recall is given by 
$$
\lambda\mapsto\inf\{-\log|a_i|+i\cdot \lambda\}.
$$
It is a continuous, piece-wise linear and concave ($\cap$-shaped) function.
\end{itemize}
If no confusion arises, we may drop writing $t$ in the index. 

\begin{defn}\label{def: radius}
 Let $a\in k$ and let $f(t)\in k[[t-a]]$. We define the radius of convergence of $f(t)$ at $a$ and denote it by $\cR_a\big(f(t)\big)$ as 
 $$
 \cR_a\big(f(t)\big):=\sup\{r\in [0,1)\mid f\in \cO_t(a,r^-)\}=\min\{1,\sup\{r\mid f(t)\in \cO_t(a,r^-)\}\}.
 $$
 
 If $\vec{f}(t)\in k[[t-a]]^n$ is a vector of power series, we define $\cR_a\big(\vec{f}(t)\big)$ as the minimum of radii of convergence of its components.
\end{defn}
We note that in general $\cR_a\big(f(t)\big)$ is different from the {\em true} radius of convergence of a power-series as soon as the latter is bigger than 1. We also note that if $r=\cR_a\big(f(t)\big)>0$ then $f(t)\in \cO_t(a,r^-)$ but it does not have to hold $f(t)\in\cO_t(a,r)$.

\subsubsection{} We briefly recall that all the points of a disc $\sD_t^\pm$ are classified in 4 types, according to their residue field (see for example \cite[Section 3.6.]{BerCoh} for details). For our purposes, we will only work with points of type 2 and 3 along with rational points (points in the base field $k$). 

We recall that if $\eta\in \sD_t^\pm$ is a point of type 2 or 3, then there exists a closed disc $\sD_t(a,r)\subseteq \sD_t^\pm$ such that $\eta$ is the Shilov (maximal) point of $\sD_t(a,r)$. Although point $a$ is not uniquely determined in this way, the radius $r$ is and we say that $r$ is the $t$-radius of the point $\eta$ and write $r_t(\eta)=r$ and $\eta_{a,r}:=\eta$, with convention that $\eta_{a,0}:=a$.

If $a\in \sD_t^-$ (resp. $\sD_t$) we denote by $l_{t,a}$ the set of points of the form $\eta_{a,r}$, where $r\in [0,1)$ (resp. $r\in [0,1]$) with inherited topology from $\sD_t$. 

\subsection{Setting}

\subsubsection{} Let $\vphi:\sD^\pm_t\to \sD^\pm_s$ be a finite morphism of unit discs of degree $d$. Note that the discs are then simultaneously open or closed (a fact that will be used in what follows without explicitly mentioning it). Then, $\vphi$ comes with a morphism $\vphi^{\#}:\cO^\pm_s\to \cO^\pm_t$ which makes $\cO^\pm_t$ a finite $\cO^\pm_s$-module. The image of $s$ generates the morphism and if we identify $s$ with $\vphi^\#(s)$ we obtain $(s,t)$-coordinate representation of $\vphi$ in the form:
\begin{equation}\label{eq: coord}
s=f(t)=\sum_{i\geq 0} f_i\cdot t^i. 
\end{equation}
Then, we have:
\begin{enumerate}
\item The point $a\in \sD_t(k)$ is mapped to a point $f(a)\in \sD_s(k)$.  
\item For each disc $\sD_t(a,r^\pm)\subset \sD_t$ we have $\vphi\big(\sD_t(a,r^\pm)\big)=\sD_s(f(a),r'^\pm)$. More precisely, the radius $r'$ is given by 
$$
-\log r'=v(f(t)-f(a),-\log r).
$$
Furthermore,  the restriction $\vphi_{|\sD_t(a,r^\pm)}$ is a finite morphism of closed (resp. open) discs. Then, if we choose coordinates $t-a$ and $s-f(a)$ on $\sD_t(a,r^\pm)$ and $\sD_s(f(a),r'^\pm)$, respectively, we obtain coordinate representation for $\vphi_{\sD_t(a,r^\pm)}$ in the form
\begin{equation}\label{eq: coord restr}
s-f(a)=f_a(t):=\sum_{j\geq 1} f_{j,a}\cdot(t-a)^j,
\end{equation}
which is nothing but the Taylor expansion of $f(t)$ from \eqref{eq: coord} at point $a$. Finally, the degree of $\vphi_{|\sD_t(a,r^\pm)}$ is equal to the right slope of $v(f(t)-f(a), \cdot )$ at $-\log r$ if the discs are open and to the left slope at $-\log r$ if $r<1$ and the discs are closed.

We note when the disc $\sD_t^-$ is open, the highest slope of $v(f(t),\dot)$ is then equal to the degree of the morphism $d$.

\item For any $\sD_s(b,r^\pm)\subset\sD^\pm_s$ the inverse image $\vphi^{-1}\big(\sD_s(b,r^\pm)\big)=\cup_{i=1}^n\sD_t(a_i,r_i^\pm)$ where the latter is a disjoint union of discs and $f(a_i)=b$.

In particular, for each $b\in \sD^\pm_s(k)$, the cardinality of $\vphi^{-1}(a)$, counting multiplicities, is $d$. 
 \end{enumerate}
 
\subsubsection{}\label{subs: poly} On the other side, it follows that there is some natural number $m$ such that 
$$
a_0(s)+a_1(s)\cdot t+\dots a_m(s)\cdot t^m=0, \quad a_i(s)\in \cO^\pm_s. 
$$
The minimal such $m$ is equal to $d$. Then $a_d(s)$ is invertible in $\cO^\pm_s$. Indeed, if $a\in \sD_s^\pm(k)$,  $a_d(a)=0$ would imply that $\vphi^{-1}(a)$ has less or equal than $d-1$ preimages in $\sD_t^\pm(k)$ which contradicts the previous point. So $a_d(a)\neq 0$ and $a_d(s)$ is invertible in $\cO_s^\pm$. Hence, there exists a unique monic polynomial $P(s,X)=a_0(s)+\dots+a_{d-1}(s)\cdot X^{d-1}+X^d\in \cO^\pm_{s}[X]$ of degree $d$ such that $P(s,t)=0$.

The point (3) above then translates in the following way. For each $b\in \sD^\pm_s(k)$ and $r\in (0,1)$, we have a natural embedding $\cO^\pm_s\hookrightarrow\cO_s(b,r^\pm)$. Then $\vphi^{-1}\big(\sD_s(b,r^\pm)\big)=\cup_{i=1}^n\sD_t(a_i,r_i^\pm)$ corresponds to a factorization of $P(s,X)\in \cO_s(b,r^\pm)[X]$ into $n$ irreducible monic polynomials $P_i(s,X)$, where $P_i(s,X)$ corresponds to the restriction $\vphi_{|\sD_t(a_i',r_i^\pm)}:\sD_t(a_i',r_i^\pm)\to \sD_s(b,r^\pm)$ and its zero is the image of $t=a_i'+(t-a_i')$ in $\cO_t(a_i',r_i^\pm)$ and where $a_1',\dots,a_n'\in \vphi^{-1}(b)$. Once we put $a=a_i$, \eqref{eq: coord restr} gives relation between images of $s$ and $t$ in $\cO_s(b,r^\pm)$ and $\cO_t(a_i,r_i^\pm)$, respectively.

 We have the following commutative diagrams
 
\begin{diagram}[H]
\centering
\begin{minipage}{0.45\textwidth}
\centering
\begin{tikzcd}
  \sD_t^\pm\arrow{r}{\vphi}  & \sD_s^\pm\\
  U:=\bigcup_{i=1}^n\sD_t(a_i,r_i^\pm)\arrow[u, hook] \arrow{r}{\vphi_{|U}}& \sD_s(b,r^\pm)\arrow[u,hook]
\end{tikzcd}
\caption{}
\label{dia: d1}
 \end{minipage}\hfill
 \begin{minipage}{0.45\textwidth}
 \centering
 \begin{tikzcd}
  \cO_t\pm \arrow[d] \arrow[r,hookleftarrow, "\vphi^\#"] & \cO_s^\pm \arrow[d]\\
  \bigoplus_{i=1}^n\cO_t(a_i',r_i^\pm) \arrow[r,leftarrow, "\vphi_{|U}^\#"]& \cO_s(b,r^\pm)
\end{tikzcd}
\caption{}
\label{dia: d2}
\end{minipage}\hfill
 \end{diagram} 
 where in Diagram \ref{dia: d1} the vertical arrows are natural inclusions while in Diagram \ref{dia: d2} the vertical arrows are restrictions.

\subsection{Tree over a point}\hfill

\begin{defn}
 Let $b\in \sD_s^\pm(k)$. The set $\cup_{a\in \vphi^{-1}(b)}l_{t,a}=\vphi^{-1}(l_{s,b})$ (with inherited topology from $\sD_t$) is called  $\vphi${\emph-tree over $b$} and we denote it by $\T_{\vphi,b}$.
 
 A point $y\in \T_{\vphi,b}$ is called branching if $\T_{\vphi,b}\setminus\{y\}$ has more than two connected components. We denote the set of branching points by $\cB_{\vphi,b}$.
 
 Let $\eta\in \cB_{\vphi,b}$. We call any open disc in $\sD\subseteq \sD_t^\pm$ a branch at $\eta$ if $\sD$ is a connected component of $\sD_t^\pm\setminus \{\eta\}$ and $\T_{\vphi,b}\cap \sD$ is nonempty. We denote by $\delta(\eta)$ the cardinality of the set of branches at $\eta$. 
\end{defn}

It is easy to see that $\eta\in \sD_t^\pm(k)$ is a branching point in $\T_{\vphi,b}$ if there exists $a_1,a_2\in \vphi^{-1}(b)$, $a_1\neq a_2$ such that $\eta=\eta_{a_1,r}=\eta_{a_2,r}$, for some $r\in (0,1]$. The set of branching points is finite.

\begin{defn}\label{defn: branching radii}
 We say that $r\in (0,1]$ is branching radius if there is a branching point $\eta\in T_{\vphi,b}$ such that $r_s\big(\vphi(\eta)\big)=r$. We also say that $r$ is the branching radius of $\eta$, and we denote by $\cB_{\vphi,b}(r)$ all the branching points of $\T_{\vphi,b}$ whose branching radius is $r$.
 
\end{defn}

In fact, if $\sD$ is any disc attached to the branching point $\eta$, then $r_s\big(\vphi(\eta)\big)$ is the radius of the disc $\vphi(\sD)$.

\begin{lemma}\label{lem: sum of branches}
Let $\sD\subseteq \sD_t^\pm$ be a disc. Then, 
$$
\sum_{\eta\in \cB_{\vphi,b}\cap \sD}(\delta(\eta)-1)+1=\#\big(\sD\cap \vphi^{-1}(b)\big).
$$
\end{lemma}
\begin{proof}
We note that if we substitute $\sD$ with a disc $\sD'\subseteq \sD$ such that $\sD\cap\vphi^{-1}=\sD'\cap \vphi^{-1}(b)$, then none of the numbers in equality change. 

In particular, we may substitute $\sD$ with the smallest closed disc $\sD'=\sD_t(a,r)$ in $\sD_t$ which contains points $\sD\cap \vphi^{-1}(b)$. If $r=0$, then $\sD'$ contains no branching points and $\#\big(\sD'\cap \vphi^{-1}(b)\big)=1$ and the equality holds. 
If $r>0$, then $\eta_{a,r}$ is in $\cB_{\vphi,b}$ and in this case we note that $\T_{\vphi,b}\cap \sD'$ has naturally a structure of a planar graph with set of vertices $V$ being the union of $\cB_{\vphi,b}\cap \sD'$ and $\sD'\cap \vphi^{-1}(b)$, and edges being the connected components of $(\T_{\vphi,b}\cap \sD)\setminus V$. Note that the number of edges is exactly $\sum_{\eta\in \cB_{\vphi,b}\cap \sD'}\delta(\eta)$.  Euler's formula then reads 
$$
\#\big(\sD'\cap\vphi^{-1}(b)\big)+\sum_{\eta\in \cB_{\vphi,b}\cap\sD'}1-\sum_{\eta\in \cB_{\vphi,b}\cap\sD'}\delta(\eta)+1=2,
$$
which implies the lemma.
\end{proof}

%
%

\subsection{Solutions}\hfill\\

\begin{defn}
 Let $b\in \sD^\pm_s(k)$. We say that $u(s)\in k[[s-b]]$ is a solution of $P(s,X)=0$ at $b$ if $P(s,u(s))=0$. 
\end{defn}

 We continue discussion in \ref{subs: poly}. Suppose that $b\in\sD_s^\pm$ is not branching for $\vphi$. Then, for a small enough disc $\sD:=\sD_s(b,r^-)$, we have  $\vphi^{-1}(\sD)=\cup_{i=1}^d\sD(a_i,r_i^-)$ where the union is disjoint and which corresponds to factorization of polynomial $P(s,X)$ into linear factors of the form $X-f_i(s)$, where $f_i(s)\in \cO_s(b,r^-)$, and which in turn corresponds to the restriction $\vphi_{|\sD_t(a_i,r_i^\pm)}:\sD_t(a_i,r_i^\pm)\to \sD$ (which is then an isomorphism). In particular, we may write $t=f_i(s)$. 
 
 Now, since $t=a_i$ is sent to $s=b$ we obtain that $f_{i,0}=a_i$, which gives $t=a_i+\sum_{j\geq 1}f_{i,j}\cdot(s-b)^j$, and the latter series has a non-zero radius of convergence.  Then, we proved 

\begin{lemma}\label{lem: solution}
 Suppose that $b\in \sD_s^\pm(k)$ is not a branching point for $\vphi$. Then, there are exactly $d$ solutions of $P(s,X)=0$ at $b$. 
 More precisely, for each $a\in \sD_s^\pm(k)$ which is a zero of $P(0,X)$ corresponds one solution of the form  $u_a(s)=a+\sum_{i\geq 1}u_{a,i}\cdot(s-b)^i$.
 
 Each solution has a non-zero radius of convergence.
\end{lemma}

\subsection{Sections}\hfill\\

The setting and notation remains as before.
\begin{defn}
 Let $U\subset \sD_s^\pm$ be a $k$-analytic subset. A section of $\vphi$ over $U$ is any morphism $\phi:U\to \vphi^{-1}(U)$ such that $\vphi\circ\phi=\id$. 
 
 A section of $\vphi$ at $b\in \sD_s^\pm(k)$ is any section of $\vphi$ over $U$, where $U$ is some open neighborhood of $b$.  
 
 Two sections at $b$ are equal if their restrictions coincide on some neighborhood of $b$.
\end{defn}

Suppose that $b\in \sD_s^\pm$ is not a branching point for $\vphi$ and let $r\in (0,1)$ be such that $\vphi^{-1}(\sD_s(b,r^-))$ has exactly $d$ connected components denoted by $\sD_t(a_i,r_i^-)$, $i=1,\dots,d$. Then, the restriction $\vphi_{|\sD_{t,i}}:\sD_t(a_i,r_i^-)\to \sD_s(b,r^-)$ is an isomorphism of open discs, and as we saw in the previous section, it has $(t-a_i,s-b)$-coordinate representation given by $t-a_i=\sum_{j\geq 1}u_{a_i,j}\cdot(s-b)^j$, coming from the solution $u_{a_i}(s)$ of $P(s,X)$ at $b$.  The latter function is then inversible in $\cO_s(b,r^-)$ and we may write $s-b=\sum_{j\geq 1}g_{i,j}\cdot(t-a_i)^j$ which is a coordinate representation of the inverse morphism $\phi_i:=\vphi_{|\sD_t(a_i,r_i^-)}^{-1}:\sD_s(b,r^-)\to \sD_s(a_i,r_i^-)$. 

In particular, morphism $\phi'_i:\sD_s(b,r^-)\to \vphi^{-1}\big(\sD_s(b,r^-)\big)$ induced by $\phi_i$ is a section of $\vphi$ over $\sD_s(b,r^-)$. At the level of functions, $\phi'$ is generated by $t_{|\sD_t(a_i,r_i^-)}\mapsto\sum_{j\geq 0}u_{a_i,j}(s-b)^j$ and $t_{|\sD_t(a_j,r_j^-)}\mapsto 0$, for $j\neq i$. 

In this way we proved
\begin{lemma}\label{lem: sections}
 Suppose that $b\in \sD_s^\pm(k)$ is not a branching point for $\vphi$. Then, there are exactly $d$ different sections of $\vphi$ at $b$. 
 
 More precisely, for each $a$ in $\vphi^{-1}(b)$, and $u_a(s)$ a solution at $b$ with $u_a(0)=a$, there is a section $\phi_a$ that is given by $t_{|\sD_t(a,r^-)}\mapsto u_a(s)$ and $t_{|U}\mapsto 0$, where $r\in (0,1)$ such that $\vphi_{|\sD_t(a,r)}$ is an isomorphism and $U$ is any connected component of $\vphi^{-1}\big(\vphi(\sD_t(a_i,r_i^-))\big)$ which is different from $\sD_t(a_i,r_i^-)$.
\end{lemma}

\begin{rmk} 
Let $\phi_a$ be a section of $\vphi$ at $b$ that corresponds to the solution $u_a(s)$. Each element $g(t)\in\cO_t^\pm$ can uniquely be written as $g(t)=g_0(s)+g_1(s)\cdot t+\dots+g_{d-1}(s)\cdot t^{d-1}$, for some elements $g_i(s)\in \cO_s^\pm$. Then, $\phi_a^\#\big(g(t)\big)$ is given by $\phi_a^\#\big(g(t)\big)=g_0(s)+g_1(s)\cdot u_a(s)+\dots+g_{d-1}(s)\cdot u_a(s)^{d-1}$.
\end{rmk}

\section{$p$-adic differential modules over discs}

\subsection{General constructions}\hfill\\
\subsubsection{}
A standard reference for this section is \cite[Chapter 5]{Kedbook}. Let $K$ be a field and let $(R,d)$ be a $K$-differential ring, that is, a (commutative) $K$-algebra $R$ equipped with an additive map (derivation) $d:R\to R$ that satisfies $d(a\cdot b)=d(a)\cdot b+a\cdot d(b)$ and such that $d(\alpha)=0$, for $\alpha\in K$. By a differential module $(M,D)$ over $(R,d)$ of rank $\fr$ we mean a finite free  $R$-module $M$ of rank $\fr$ equipped with an additive map (derivation) $D:M\to M$ which satisfies $D(a\cdot m)=d(a)\cdot m+a\cdot D(m)$, for every $a\in R$ and $m\in M$. In particular, $(R,d)$ is a differential module over itself. We put $M^D:=\ker D$ and note that the latter is a $K$-vector space and we call its elements {\em horizontal}.

A morphism between two differential modules $(M_1,D_1)$, $(M_2,D_2)$ over $(R,d)$ is a morphism of $R$-modules $f:M_1\to M_2$ which commutes with derivations, $f\circ D_1=D_2\circ f$. Such a map is sometimes called horizontal. Note that if $(M_1,D_1)$ and $(M_2,D_2)$ are isomorphic differential modules, the $K$-vector spaces $M_1^{D_1}$ and $M_2^{D_2}$ are isomorphic.

Many operations that one can perform on $R$-modules carry on to differential modules (with suitable modifications). We recall some of them that will serve for our purposes.

Let $(M_1,D_1)$ and $(M_2,D_2)$ be differential modules over $(R,d)$ of ranks $\fr_1$ and $\fr_2$, respectively, and $m_1\in M$ and $m_2\in M_2$. Then, $(M_1\oplus M_2,D_1\oplus D_2)$, with $(D_1\oplus D_2)(m_1\oplus m_2)=D_1(m_1)\oplus D_2(m_2)$, and $(M_1\otimes M_2, D_1\otimes D_2)$ with $D_1\otimes D_2(m_1\otimes m_2)=D_1(m_1)\otimes m_2+m_1\otimes D_2(m_2)$ are also differential modules over $(R,d)$. Their ranks are $\fr_1+\fr_2$ and $\fr_1\cdot \fr_2$, respectively.

Let $\fr$ be the rank of $M$ and let $\ue:=\{e_1,\dots,e_\fr\}$ be an $R$-basis for $M$. If we write $D(e_i)=a_{i,1}\cdot e_1+\dots+a_{i,\fr}\cdot e_\fr$ we obtain an $\fr\times \fr$ matrix $A_\ue:=(a_{i,j})\in M_\fr(R)$. We say that $A_\ue$ is the {\em derivation matrix} with respect to $\ue$.

In particular, if we take $m=v_1\cdot e_1+\dots+v_\fr\cdot e_\fr\in M$, and identify $m$ with the column vector $\vec{v}:=(v_1,\dots,v_\fr)^T$ then  $D(m)=D(\vec{v})=d(\vec{v})+A^T_\ue\vec{v}$, which completely describes the action of $D$ on $M$ with respect to basis $\ue$. For example,
\begin{equation}\label{eq: assign system}
m\in M^D \quad \text{if and only if }\quad d(\vec{v})=-A_\ue^T\vec{v}, 
\end{equation}
hence to find horizontal elements one is led to solve the differential system on the right-hand side of \eqref{eq: assign system}. We call this system the {\em associated system} to differential module $(M,D)$ with respect to basis $\ue$.

We will often use identification of $m$ with the column vector $\vec{v}_\ue:=\vec{v}$ and may drop writing $\ue$ in the index if the basis is fixed or known from the context.

If now $\ue'$ is another $R$-basis for $M$, then there is a matrix $B:=B_{\ue,\ue'}\in Gl_\fr(R)$ such that $\ue=B\ue'$ and a direct calculation shows that the derivation matrix of $(M,D)$ with respect to $\ue'$ is 
\begin{equation}\label{eq: change basis}
A_{\ue'}=d(B_{\ue',\ue}) B_{\ue,\ue'}+B_{\ue',\ue}A_\ue B_{\ue,\ue'}.
\end{equation}
We finally recall that $\dim_K M^D\leq r$.
\subsubsection{}

Let $(R,d)$ and $(R',d')$ be differential rings and let $\phi:R\to R'$ be a $K$-algebra homomorphism such that there exists some $r'\in R'^*$ so that $d'(\phi(r))=r'\cdot\phi(d(r))$, for all $r\in R$. Let $(M,D)$ be a differential module over $(R',d')$. Then, we put $M_\phi$ to denote $R$-module obtained from $M$ by restriction of scalars via $\phi$, and $D_\phi$ a derivation on $M$ defined by $D_\phi(m)=\frac{1}{r'}\cdot D(m)$. For $m\in M$ we will often write $(m)_\phi$ if we want to see $m$ as an element of $M_\phi$. 

\begin{defn}
 If $R'$ is a finite $R$-module, then $(M_\phi,D_\phi)$ is a differential module over $(R,d)$ which is called a direct image of $(M,D)$ by $\phi$.
\end{defn}

In fact, it is not hard to see that if $e_1,\dots,e_\fr$ is a basis of $M$ over $R'$ and $f_1,\dots,f_n$ is a basis of $R'$ over $R$, then $f_ie_j$, $i=1,\dots n$, $j=1,\dots,\fr$ is a basis of $M_\phi$ over $R$.  

\begin{rmk}\label{rmk: bijection horizontal}
We note that there is a $K$-linear bijection between $M^D$ and $_\phi M^{D_\phi}$ by construction which gives us that $\dim_K M^D=\dim_K {_\phi} M^{D_\phi}$. 
\end{rmk}
\begin{lemma}\label{lem: ext restr}
 Let $(M,D)$ be a differential module over $(R',d')$, and let 
 \begin{diagram}[H]
\centering
\begin{tikzcd}
  R' \arrow[d,hook,"\psi'"] \arrow[r,leftarrow, "\phi"] & R \arrow[d,hook,"\psi"]\\
  R'_0 \arrow[r,leftarrow, "\phi_0"]& R_0
\end{tikzcd}
 \end{diagram}
be a commutative diagram where $(R',d'), (R_0,d_0), (R_0',d_0')$ are differential rings, $\phi,\phi_0,\psi, \psi'$ are $K$-algebra morphisms, $\phi$ and $\phi_0$ are finite of the same degree, and $\psi$ and $\psi'$ are horizontal. Suppose further there exists an invertible element $\alpha\in R'$ such that for any $a\in R$ (resp. $a_0\in R_0$), we have $d'(\phi(a))=\alpha\cdot\phi\big(d(a)\big)$ (resp. $d'(\phi_0(a_0))=\psi'(\alpha)\cdot\phi_0\big(d_0(a_0)\big)$). Then, we have an isomorphism of differential modules
$$
(M_{\phi}\otimes R_0, D_\phi\otimes d_0)\iso \big((M\otimes R_0')_{\phi_0},(D\otimes d_0)_{\phi_0}\big).
$$
\end{lemma}

\begin{proof}
 
An element of the module on the left-hand side has a form $\sum_{i=1}^l(m_i)_{\phi}\otimes a_i$, for some $m_i\in M$ and $a_i\in R_0$. Then, it is easy to see that the map
$$
\sum_{i=1}^l(m_i)_{\phi}\otimes a_i\mapsto \left(\sum_{i=1}^lm_i\otimes \phi_0(a_i)\right)_{\phi_0}
$$
gives us the required horizontal isomorphism.
\end{proof}

\subsection{Setting}\hfill\\

\subsubsection{}
In our setting, the base differential rings will be of the form $(\cO_t^\pm,d_t)$ (with $K=k$) where $d_t=\frac{d}{dt}$. If $(M,D)$ is a differential module over $(\cO_t^\pm,d_t)$, we put $D_t:=D$ to keep track of variables. 

Furthermore, if $a\in \sD_t^\pm(k)$ and $r\in (0,1)$ we denote by $(M_{a,r^\pm},D_t)$ the differential module $(M\otimes \cO_t(a,r^\pm),D_t\otimes \frac{d}{dt})$, hoping that it will be clear from the context whether by $D$ we mean $D_t$ or $D_t\otimes \frac{d}{dt}$.

Let $(M,D_t)$ be a differential module of rank $\fr$ over $(\cO_t^\pm,d_t)$. For $i=1,\dots,\fr$, we put
$$
\cR_i:=\cR_i\big(a,(M,D)\big):=\sup\{s\in (0,1]\mid \dim_k M_{a,r^-}^{D_t}\geq \fr-i+1\}.
$$

\begin{defn}
 The $r$-tuple $(\cR_1,\dots,\cR_\fr)$ is called the {\em multiradius} of convergence of (horizontal elements of) $(M,D_t)$ at $a$. The number $\cR_1$ is called the radius of convergence of (horizontal elements of) $(M,D_t)$ at $a$. 
\end{defn}

That the definition is good follows from the following, $p$-adic Cauchy theorem.

\begin{prop}
 The radius of convergence of $(M,D)$ at $a$ is a non-negative number. In particular, for every  $r\in \big(0,\cR_1(a,(M,D))\big)$, $\dim_k M_{a,r}^D=\fr$.  
\end{prop}
Having in mind the proposition, in what follows we will use $(M_a,D_t)$ to denote the differential module $(M_{a,r^-},D_t)$, for some unspecified $r\in \big(0,\cR_1(a,(M,D_t))\big)$. Then, it makes sense to speak of a basis of $M_a^{D_t}$ to which we will also refer as to a {\em basis of horizontal elements of $(M,D_t)$ (or a basis of $M^{D_t}$) at $a$}.

In fact, one can as well define radius of convergence at $a$ of a single horizontal element $m\in M_a^{D_t}$. For this, let $\ue$ be a basis for $M$ over $\cO_t^\pm$ and let $A_\ue:=A_\ue(t)\in M_\fr(\cO_t^\pm)$ be the matrix of derivation $D_t$, with respect to $\ue$. Then, $m$ can be identified with the column vector of analytic functions $\vec{y}(t)\in k[[s-a]]^\fr$ and we put $\cR_a(m):=\cR_a\big(\vec{y}(t)\big)$. That this definition does not depend on the chosen basis $\ue$ follows from a simple observation that if $B(t)\in Gl_\fr(\cO_t^\pm)$, then $\cR_a\big(B(t)\,\vec{y}(t)\big)\geq \cR_a\big(\vec{y}(t)\big)$ (recall our Definition \ref{def: radius} and equations \eqref{eq: assign system} and \eqref{eq: change basis}).

\begin{defn}\label{defn: optimal}
A basis $m_1,\dots,m_\fr$ of space of horizontal elements of $(M,D)$ at $a$ is called optimal if numbers $\cR_a(m_1),\cdots,\cR_a(m_\fr)$ can be rearranged to form the multiradius of convergence of $(M,D)$ at $a$. 
%
%
\end{defn}
Another way to look at optimal basis is to consider 
$$
\Pi:=\sup\{\prod_{i=1}^\fr\cR_a(m_i)\mid m_1,\dots,m_\fr\text{ is a basis of $M^D$ at $a$ }\}.
$$
Then, $m_1,\dots,m_\fr$ is optimal basis for $M^D$ at $a$ if and only if (see \cite{You})
$$
\Pi=\prod_{i=1}^\fr\cR_a(m_i).
$$

Here is a simple but useful criterion for a basis of horizontal elements to be optimal.

\begin{lemma}\label{lem: optimal criterion}
Let $(M,D_t)$ be a differential module of rank $\fr$ over $\cO_t^\pm$, and let $m_1,\dots,m_\fr$ be a basis of horizontal elements of $M$ at some $a\in \sD_t^\pm(k)$. The following are equivalent:
\begin{enumerate}
 \item $m_1,\dots,m_\fr$ is an optimal basis at $a$;
 \item let $J_r\subseteq \{1,\dots,\fr\}$ be such that $\cR_a(m_j)=r$ for any $j\in J_r$. Then, the radius of convergence of any nontrivial $k$-linear combination 
 $
 \sum_{j\in J_r}\alpha_j\cdot m_j
 $
 is $r$. 
 \item If $m\in M^D_a$ has radius of convergence $r$, then $m=\sum_{i=1}^\fr\beta_i\cdot m_i$ with $\cR_a(m_i)\geq r$ for every $i$ with $\beta_i\neq 0$ and for at least one such $i$ we have $\cR_a(m_i)=r$.
\end{enumerate}

\end{lemma}
\begin{proof}
\emph{(1)}$\Rightarrow$\emph{(2)}  It is clear that the radius of convergence of $m_0:=\sum_{j\in J}\alpha_j\cdot m_j$ is bigger than or equal to $r$. Suppose it is strictly bigger and $r<1$. Then, if we take out from $m_1,\dots,m_\fr$ one $m_{j_0}$ with $j_0\in J_r$ and $\alpha_{j_0}\neq 0$ and substitute it with $m_0$, we will obtain a new basis of horizontal elements, but this time with the product of the radii of convergence strictly bigger than that of the starting basis, which is is a contradiction.

\emph{(2)}$\Rightarrow$\emph{(1)} Suppose that the basis is not optimal, and let $m_1',\dots,m'_\fr$ be any optimal basis of horizontal elements of $M$ at $a$. We note that by our assumption and first part of the proof, for both bases $\{m_1,\dots,m_\fr\}$ and $\{m_1',\dots,m_\fr'\}$ property \emph{(2)} holds. 

From Definition \ref{defn: optimal} it follows that there is some $r\in (0,1]$ such that the number of $m_i$'s which have radius of convergence $r$ is strictly smaller than the number of $m'_i$'s with the same property, which implies that the number of $m_i$'s which have radius of convergence strictly bigger than $r$ is smaller than the number of $m'_i$'s with the same property. Hence, there is some $m'_j$ with radius of convergence bigger than $r$ such that if we write 
\begin{equation}\label{eq: first sum}
m'_j=\sum_{i=1}^\fr \alpha_i\cdot m_i,
\end{equation}
there will be some $i_0\in \{1,\dots,\fr\}$ such that  $\alpha_{i_0}\neq 0$ and $\cR_a(m_{i_0})=r_0$. We may choose $i_0$ so that $r_0$ is smallest among all the radii of convergence of elements $m_i$ for which $\alpha_i\neq 0$. Note that by what we said before $r_0\leq r$. Then, we may rewrite \eqref{eq: first sum} as 
$$
\sum_{\substack{i=1\\ \cR_a(m_i)=r_0}}^\fr\alpha_i\cdot  m_i=-\sum_{\substack{i=1\\\cR_a(m_i)>r_0}}^\fr\alpha_i\cdot m_i-m_j',
$$
where the sum on the left is non-trivial. Howerever, the right-hand side has radius of convergence bigger than $r_0$ which contradicts assumption \emph{(2)}.

\emph{(2)}$\Rightarrow$\emph{(3)} Let $r_0$ be the smallest among the radii of convergence of $m_i$ for which $\beta_i\neq 0$. Then, we may write
$$
\sum_{\substack{i=1\\ \cR_a(m_i)=r_0}}^\fr\beta_i\cdot  m_i=-\sum_{\substack{i=1\\\cR_a(m_i)>r_0}}^\fr\beta_i\cdot m_i-m,
$$
hence by \emph{(2)} it follows that $r_0=r$ which implies \emph{(3)}. 

\emph{(3)}$\Rightarrow$\emph{(2)} is easy.
\end{proof}


\subsubsection{}\label{ssec: setting}

Let $\vphi:\sD_t^\pm\to \sD_t^\pm$ be a finite \'etale morphism of unit discs of degree $d$, and $s=f(t)$ its $(s,t)$-coordinate representation. Then $f'(t)$ is invertible in $\cO_t^\pm$ and we have that for any $g(s)\in \cO_s^\pm$, $d_t\big(g(f(t))\big)=f'(t)\cdot\vphi^\#(d_s(g(s)))$ so in particular, we may form the direct image of $(M,D_t)$ by $\vphi$ which we denote by $(M_\vphi,D_s)$.

We recall that the action of $D_s$ on $M_\vphi$ is then given by $D_s(m)=\big(\frac{1}{f'(t)}\cdot D_t(m)\big)_\vphi$. More precisely, if $\{e_1,\dots,e_\fr\}$ is a $\cO_t^\pm$ basis for $M$, and if $m\in M$, $m=a_1(t)\cdot e_1+\dots+a_\fr(t)\cdot e_\fr$, then
$$
D_s(m)=\left(\frac{1}{f'(t)}\cdot\sum_{i=1}^\fr\left(\frac{d}{dt}a_i(t)\cdot e_i+a_i(t)\cdot D_t(e_i)\right)\right)_\vphi=\sum_{i=1}^\fr\sum_{j=0}^{d-1}a_{i,j}(s)\cdot t^j(e_i)_\vphi.
$$

Let $r\in (0,1)$, $b\in \sD_s(b,r^-)$ and put $\{a_1,\dots,a_d\}=\vphi^{-1}(b)$ and let $U:=\vphi^{-1}\big(\sD_s(b,r^-)\big)=\cup_{i=1}^n\sD_t(a_i',r_i^-)$, where the union is disjoint and where $\vphi(a_i')=b$. We recall that we have a commutative Diagram \ref{dia: d2}, and we note that the restriction morphisms $\psi_s:(\cO_s^\pm,d_s)\to (\cO_s(b,r^-),d_s)$ and $\psi_t:(\cO_t^\pm,d_t)\to (\oplus_{i=1}^n\cO_t(a_i',r_i^-),\oplus_{i=1}^nd_t)$  are  horizontal morphisms of differential rings. Furthermore, for every $g(s)\in \cO_s(b,r^-)$, we have that 
$$
d_t\big(\vphi_{|U}^\#(g(s))\big)=\bigoplus\limits_{i=1}^n d_t\big(g(f(t))\big)=f'(t)\cdot\bigoplus_{i=1}^ng'\big(f(t)\big)=f'(t)\cdot\vphi_{|U}^{\#}\big(d_s(g(s))\big),
$$
hence all the conditions of Lemma \ref{lem: ext restr} are satisfied and we obtain the following
\begin{cor} \label{cor: iso diff}
Let $(M,D_t)$ be a differential module over $(\cO_t^\pm,d_t)$. Then,  the map 
\begin{align*}
 \Phi_r:M_\vphi\otimes \cO_s(b,r^-)&\to\big(M\otimes\bigoplus_{i=1}^n\cO_t(a_i',r_i^-)\big)_{\vphi_{|U}}=\bigoplus_{i=1}^n\big(M\otimes \cO_t(a_i',r_i^-)\big)_{\vphi_{|\sD_t(a_i',r_i^-)}}\\
 \sum_{j=1}^l (m_j)_\vphi\otimes g_j(s) &\mapsto \left(\sum_{j=1}^l m_i\otimes\vphi_{|U}^\#\big(g_j(s)\big)  \right)_{\vphi_{|U}} =\bigoplus_{i=1}^n\sum_{i=1}^l\left( m_i\otimes \vphi_{|\sD_t(a_i',r_i^-)}^{\#}\big(g_i(s)\big)\right)_{\vphi_{|\sD_t(a_i',r_i^-)}},
\end{align*}
induces an isomorphism of differential modules.
\end{cor}

\subsubsection{}\label{ssec: setting cont} We next describe more precisely isomorphism $\Phi_r$. Let us denote its inverse by $\Psi_r$ and keep the setting and notation from the previous section. Let us fix once and for all an $\cO_t^\pm$-basis $\ue=\{e_1,\dots,e_\fr\}$ for $M$. Then, any element $m$ in $M_\vphi\otimes \cO_s(b,r^-)$ can be written as (we identify $s$ with $\vphi^\#(s)$)
$$
m=\sum_{j=1}^\fr\sum_{m=0}^{d-1}g_{j,m}(s)\cdot t^m\cdot e_j,\quad g_{j,m}(s)\in\cO_s(b,r^-).
$$
On the other side, the image of $\sum_{m=0}^{d-1}g_{j,m}(s)\cdot t^m$ by $\Psi_r$ defines an element in $\bigoplus_{i=1}^n \cO_t(a_i',r_i^-)$, hence an analytic function on every $\sD_t(a_i',r_i^-)$, $i=1,\dots,n$. If we denote the degree of $\cO_s(b,r^-)\hookrightarrow\cO_t(a_i,r_i^-)$ by $d_i$, it follows that there exist functions $G_{i,j,m}(s)\in\cO_s(b,r^-)$ such that  
\begin{equation}\label{eq: chinese remind 1}
\sum_{m=0}^{d-1}g_{j,m}(s)\cdot X^m\equiv \sum_{m=0}^{d_i-1}G_{i,j,m}(s)\cdot X^m\mod P_i(s,X),\quad i=1,\dots,n,
\end{equation}
hence
$$
\left(\sum_{m=0}^{d-1}g_{j,m}\cdot t^m\right)_{|\sD_{t}(a_i,r_i^-)}=\sum_{m=0}^{d_i-1}G_{i,j,m}(s)\cdot t^m,\quad i=1,\dots,n.
$$
Hence, with respect to corresponding bases $\{e_1,\dots,e_\fr\}$ and $\{e_1,\dots,t^{d-1}e_1,\dots,t^{d-1}e_\fr\}$, the action of $\Phi_r$ is given by 

\begin{align}\label{eq: Phi_r explicit general}
 m&\mapsto \Phi_r(m)\nonumber\\
 \sum_{j=1}^\fr\sum_{m=0}^{d-1}g_{j,m}(s)\cdot t^me_j&\mapsto \bigoplus_{i=1}^n\left(\sum_{j=1}^\fr\sum_{m=0}^{d_i-1}G_{i,j,m}(s)\cdot t^me_j\right).
\end{align}

In the other direction, to obtain $\Psi_r$, we note that given a collection of functions $G_{i,j,m}(s)$, one can recover $g_{j,m}(s)$ from relations \eqref{eq: chinese remind 1} thanks to Chinese reminder theorem. 

 \subsection{Radius of convergence and direct images}\label{sec: radius relations}\hfill\\
 
 We next use Corollary \ref{cor: iso diff} to study radius of convergence of horizontal solutions of the direct image $(M_\vphi,D_s)$ at a rational point $b$. Let us fix $r,R\in (0,1]$ with $r<R$, and let us put $\vphi^{-1}\big(\sD_s(b,r^-)\big)=\cup_{i=1}^{n}\sD_t(a_i',r_i^-)$ and $\vphi^{-1}\big(\sD_s(b,R^-)\big)=\cup_{i=1}^{n'}\sD_t(a_i'',R^-)$, both unions being disjoint. As usual, we assume that $\{a_1',\dots,a_n',a_1'',\dots,a_{n'}''\}\subset\vphi^{-1}(b)$.
 
 Let $\sD_t(a_{i,1}',r_{i,1}^-),\dots,\sD_t(a_{i,j(i)}',r_{i,j(i)}^-)$ be those discs among $\sD_t(a_i',r_i^-)$, $i=1,\dots,n$, that are contained in $\sD_t(a_i'',R_i^-)$.
 
 Then, we have the following commutative diagram of differential modules 
 
 \begin{diagram}[H]
\centering
\begin{tikzcd}
  \bigoplus\limits_{i=1}^{n'}\big(M\otimes \cO_t(a_i'',R_i^-)\big)_{\vphi_{|\sD_t(a_i'',R_i^-)}} \arrow[d,hook] \arrow[r, "\Psi_R"] & M_\vphi\otimes \cO_s(b,R^-) \arrow[d,hook,]\\
  \bigoplus\limits_{i=1}^{n}\big(M\otimes \cO_t(a_i',r_i^-)\big)_{\vphi_{|\sD_t(a_i',r_i^-)}} \arrow[r, "\Psi_r"]\arrow[d]& M_\vphi\otimes \cO_s(b,r^-)\arrow[d, "\id"]\\
  \bigoplus\limits_{i=1}^{n'}\bigoplus\limits_{l=1}^{j(i)}\big(M\otimes \cO_t(a_{i,l}',r_{i,l}^-)\big)_{\vphi_{|\sD_t(a_{i,l}',r_{i,l}^-)}} \arrow[r]& M_\vphi\otimes \cO_s(b,r^-)
\end{tikzcd}
 \end{diagram}
where the upper vertical arrows are restrictions and where the left lower vertical arrow comes from just rearranging the terms in the direct sum. 

Suppose now that $(m)_\vphi$ is a horizontal element of the direct image $(M_\vphi,D_s)$ at $b$ and that $\cR_b\big((m)_\vphi\big)\geq r$. Then, there are horizontal elements $m_{i,l}\in M\otimes \cO_t(a_{i,l}',r_{i,l}^-))$, $i=1,\dots,n'$, $l=1,\dots,j(i)$ such that 
$$
\big(\bigoplus\limits_{i=1}^{n'}\bigoplus\limits_{l=1}^{j(i)}m_{i,l}\big)_\vphi=(m)_\vphi.
$$
We remark that each $m_{i,l}$ has radius of convergence bigger than or equal to $r_{i,l}$. Similarly, if $\cR_b((m)_\vphi)\geq R$, then there are horizontal elements $m_i\in M\otimes \cO_t(a_i'',R_i^-)$, $i=1,\dots,n'$ such that 
$$
\big(\bigoplus\limits_{i=1}^{n'}m_i\big)_\vphi=(m)_\vphi.
$$
We further note that in this case that $m_{i,l}={m_i}_{|M\otimes \cO_t(a_{i,l}',r_{i,l}^-)}$, where the latter is the image of $m_i$ under restriction map (we will simply the restriction of $m_i$ to $\sD_t(a_{i,l}',r_{i,l}^-)$). In particular, each $m_{i,l}$ has radius of convergence at $a_{i,l}'$ bigger than or equal to $R_i$, and in particular bigger than $r_{i,l}$. 

In this way we proved

\begin{lemma}\label{lem: radius crit}
 The following are equivalent:
 \begin{enumerate}
 \item $\cR_{b}((m)_\vphi)=r$;
 \item For every $1\geq R>r$, there exists an $i=1,\dots,n'$ such that there does not exist a horizontal element  $m_i\in M\otimes \cO_t(a_i',R_i^-)$ such that its restriction to $\sD_t(a_{i,l}',r_{i,l}^-)$ for all $l=1,\dots,j(i)$ is $m_{i,l}$.
 \end{enumerate} 
\end{lemma}

\subsection{Explicit calculations with respect to bases}\hfill\\

\subsubsection{}\label{ssec: detail coordinates}

From now on, we fix a basis $\underline{e}:=\{e_1,\dots,e_\fr\}$ of $M$ and the corresponding basis $\ue_\vphi:=\{e_1,\dots,t^{d-1}e_1,e_2,\dots,e_\fr,\dots,t^{d-1}e_\fr\}$ of $M_\vphi$. We will write coordinates with respect to these basis as column vectors.


Let $r,R\in (0,1)$ be such that we have $\vphi^{-1}\big(\sD_s(b,r^-)\big)=\cup_{i=1}^n\sD_t(a_i',r_i^-)$ and  $\vphi^{-1}\big(\sD_s(b,R^-)\big)=\cup_{i=1}^d\sD_t(a_i,R_i^-)=:U$ (the unions, as usually, being disjoint). We clearly have $n\leq d$, where we recall that $d$ is the degree of $\vphi$. 


We note that we have the restriction morphism of differential modules
\begin{equation}\label{eq: res map}
M\otimes\bigoplus_{i=1}^n\cO_t(a_i',r_i^-)=\bigoplus_{i=1}^nM\otimes\cO_t(a_i',r_i^-)\to M\otimes\bigoplus_{i=1}^d\cO_t(a_i,R_i^-)=\bigoplus_{i=1}^dM\otimes\cO_t(a_i,R_i^-).
\end{equation}
%

Let $m\in \bigoplus_{i=1}^nM\otimes\cO_t(a_i',r_i^-)$. Then we can identify $m$ with $\vec{y'}(t)=\vec{y'}_1(t)\oplus\dots\oplus\vec{y'}_n(t)$, where $\vec{y'}_i(t)$ is a column vector of $\fr$ analytic functions which are coordinates of $m$ restricted to $M_{a_i',r_i^-}$ with respect to basis $\underline{e}$. Then, the image of $m$ under the map \eqref{eq: coord restr} can be identified with some $\vec{y}(t)=\vy_1(t)\oplus\dots\oplus\vy_d(t)$. We want to find the coordinates of $\big(\vec{y}(t)\big)_{\vphi_{|U}}$ with respect to basis $\ue_\vphi$.

For this, we write
\begin{equation}\label{eq: zero step}
\vec{y}(t)=\sum_{j=1}^\fr y_j(t)\cdot e_j, \quad \text{where}\quad y_j(t)=\oplus_{i=1}^dy_{j,i}(t)\in \oplus_{i=1}^d\cO_t(a_i,R_i^-).
\end{equation}
On the other side, we have $(\vec{y}(t))_{\vphi_{|U}}=\sum_{i=1}^d\sum_{j=1}^\fr a_{i,j}(s)\cdot t^{i-1}\cdot e_j$,  for some analytic functions $a_{i,j}(s)\in \cO_s(b,R^-)$ which are to be determined. In particular, we have
$$
\vec{y}(t)=\sum_{j=1}^\fr\left(\sum_{i=1}^da_{i,j}(s)\cdot t^{i-1}\right)\cdot e_j,
$$
that is,
\begin{equation}\label{eq: 1st step}
y_j(t)=\sum_{i=1}^d a_{i,j}(s)\cdot t^{i-1}.
\end{equation}
For each $l=1,\dots,d$, let $\phi_{a_l}:\oplus_{i=1}^d\cO_t(a_i,R_i^-)\to \cO_s(b,R^-)$ be a section of the morphism $\vphi$ at $b$ that corresponds to $a_l$, as in Lemma \ref{lem: sections}. Applying $\phi_{a_l}$ to \eqref{eq: 1st step} and using \eqref{eq: zero step} we obtain 
$$
\sum_{i=1}^da_{i,j}(s)\cdot u_{a_l}(s)^{i-1}=y_{j,l}\big(u_{a_l}(s)\big).
$$
We conclude 
$$
\bU(s)
\begin{bmatrix}
 a_{1,j}(s)\\
 a_{2,j}(s)\\
 \vdots\\
 a_{d,j}(s)
\end{bmatrix}
=\begin{bmatrix}
  y_{j,1}\big(u_{a_1}(s)\big)\\
    y_{j,2}\big(u_{a_2}(s)\big)\\
    \vdots\\
      y_{j,d}\big(u_{a_d}(s)\big)
 \end{bmatrix},
$$
where we put 
$$
\bU(s):=\begin{bmatrix}
 1 & u_{a_1}(s) & \dots & u_{a_1}(s)^{d-1}\\
 1 & u_{a_2}(s) & \dots & u_{a_2}(s)^{d-1}\\
 \vdots & \vdots & \ddots & \vdots \\
 1 & u_{a_d}(s) & \dots & u_{a_d}(s)^{d-1}
\end{bmatrix}.
$$
Being of Vandermonde type matrix $\bU(s)$ is invertible (at least locally in a neighborhood of point $b$) and if we  put $\bV(s):=\bU(s)^{-1}$ we finally obtain
$$
\begin{bmatrix}
 a_{1,j}(s)\\
 a_{2,j}(s)\\
 \vdots\\
 a_{d,j}(s)
\end{bmatrix}=\bV(s)\begin{bmatrix}
  y_{j,1}\big(u_{a_1}(s)\big)\\
    y_{j,2}\big(u_{a_2}(s)\big)\\
    \vdots\\
      y_{j,d}\big(u_{a_d}(s)\big)
 \end{bmatrix}.
$$

\begin{lemma}\label{lem: new coord}
 The element $(\vec{y}(t))_{\vphi_{|U}}$ expressed with respect to basis $\ue_\vphi$ is the vector
 $$
 \bV_\fr(s)\,\Big[y_{1,1}\big(u_{a_1}(s)\big),\dots,y_{1,d}\big(u_{a_d}(s)\big),y_{2,1}\big(u_{a_1}(s)\big),\dots,y_{r,d}\big(u_{a_d}(s)\big)\Big]^T,
 $$
 where $\bV_\fr(s)=\bigoplus_{i=1}^\fr\bV(s)$, where the latter is the $\fr$-fold direct sum (in the sense of matrices) of $\bV(s)$ with itself.  
 \end{lemma}
 
 \begin{rmk}
  We note that in concrete terms we have
 $$
 \bV_\fr(s)=\begin{bmatrix}
 \uV & \underline{0}_{d,d}& \dots &\underline{0}_{d,d}\\
 \underline{0}_{d,d}& \uV & \dots & \underline{0}_{d,d}\\
 \vdots & \vdots & \ddots & \vdots\\
 \underline{0}_{d,d} & \underline{0}_{d,d} & \dots & \uV
\end{bmatrix},
 $$
 is an $\fr\cdot d\times \fr\cdot d$ matrix where we write $\underline{0}_{i,j}$ for $i\times j$ block of zeroes, and $\uV$ for the block of entries of $\bV(s)$.
 \end{rmk}

\subsection{Basis for the space of horizontal elements for direct image}\label{ssec: horizontal}\hfill\\

 We keep the setting and notation as in Sections \ref{ssec: setting} and \ref{ssec: setting cont}. 

\begin{thm}\label{thm: fund sol}
 Let $\vec{Y}_{a_i,1}(t),\dots,\vec{Y}_{a_i,\fr}(t)$ be a basis of $M^{D_t}_{a_i}$ and let $\mathbf{Y}_i(t)$ be $\fr\times \fr$ matrix whose $j$-th column is vector $\vec{Y}_{a_i,j}(t)$. Then, a basis for the space of horizontal elements of $(M_{\vphi},D_s)$ at $b$ is given by the column vectors of the matrix
 \begin{equation}\label{eq: fund sol}
 \bV_\fr(s)\, \bigoplus_{i=1}^d\mathbf{Y}_i(u_{a_i}(s)),
 \end{equation}
 where $\bigoplus_{i=1}^d$ denotes the direct sum of the matrices involved.
\end{thm}

\begin{proof}
 It is easy to see that the columns of $\bigoplus_{i=1}^d\mathbf{Y}_i(t)$ are linearly independent which implies that columns of \eqref{eq: fund sol} are linearly independent as well since $\bV(s)$ (and then $\bV_\fr(s)$) is invertible. 
 
 In light of Lemma \ref{lem: new coord} it is enough to prove that columns of $\bigoplus_{i=1}^d\mathbf{Y}_i(t)$ are horizontal elements of $\big(M\otimes\oplus_{i=1}^d\cO_t(a_i,r_i^-),D_t\otimes\oplus_{i=1}^d d_t\big)$ which can be checked directly. 
\end{proof}

\subsection{Basis for direct image of a trivial $p$-adic differential module}\label{sec: trivial}\hfill\\

It is worth noting a special and the most simple case of Theorem \ref{thm: fund sol}.

\begin{defn}
Let $T:=\cO_t^\pm \cdot e$ and $D_t$ a derivation on $T$ given by $D_t(e)=0 \cdot e$. We say that $(T,D_t)$ is a trivial differential module (of rank $1$). 
\end{defn}
Then, obviously, $1\cdot e$ is horizontal element for $(T,D_t)$ and for any $\sD_t(a,r^-)\subseteq \sD_t^\pm$, $1\,e_{|\sD_t(a,r^-)}$ is the basis for the space of horizontal elements of $M_{a,r^-}$. In particular, the basis for $T^{D_t}_a$ at any $a\in \sD_t^\pm(k)$ is given by (the restriction of) $1\cdot e$. 

Theorem \ref{thm: fund sol} then implies
\begin{cor}\label{cor: sol trivial}
 A basis of $(T_\vphi)_b^{D_s}$ is given by the columns of the matrix $\bV(s)$.
\end{cor}

\subsection{Linked optimal bases}\hfill\\

We continue Section \ref{ssec: horizontal}, and keep the notation from Theorem \ref{thm: fund sol}. That is, we have for each $i=1,\dots,d$, $\vec{Y}_{a_i,1}(t),\dots,\vec{Y}_{a_i,\fr}$ a basis of horizontal elements of $(M,D_t)$ at $a_i$. Suppose further that it is optimal and that $\cR_{a_i}\big(\vec{Y}_{a_i,1}(t)\big)\leq\dots\leq\cR_{a_i}\big(\vec{Y}_{a_i,r}(t)\big)$.

\begin{defn}\label{defn: linked}
 We say that the optimal bases $\vec{Y}_{a_i,1}(t),\dots,\vec{Y}_{a_i,\fr}(t)$, $i=1,\dots,d$ are \emph{linked} if for each $i,j,l$ such that $|a_j-a_i|<\cR_{a_i}\big(\vec{Y}_{a_i,l}(t)\big)$, $\vec{Y}_{a_i,l}(t)$ is equal to some of $\vec{Y}_{j,1}(t),\dots,\vec{Y}_{j,\fr}(t)$.
\end{defn}

\begin{lemma}
 Linked optimal bases exist.
\end{lemma}
\begin{proof}
 We start with say $a_1$ and consider any optimal base $\vec{Y}_{1,1}(t),\dots,\vec{Y}_{1,\fr}(t)$ (with non-decreasing radii of convergence) at $a_1$. For each $j$ such that there exists some $l$ with $|a_1-a_j|<\cR_{a_1}\big(\vec{Y}_{1,l}(t)\big)$, $\vec{Y}_{1,l}(t)$ is also a horizontal element of $(M,D_t)$ at $a_j$, and so are all $\vec{Y}_{1,l+1}(t),\dots,\vec{Y}_{1,\fr}(t)$. In that case we simply take $\vec{Y}_{j,l}(t):=\vec{Y}_{1,l}(t)$. 
 
 Let $l'$ be the minimal such that $|a_1-a_j|<\cR_{a_1}\big(\vec{Y}_{1,l'}(t)\big)$. Then, we note that there exists no horizontal element $\vec{Y}(t)$ at $a_j$ such that $|a_1-a_j|<\cR_{a_j}\big(\vec{Y}(t)\big)$ and which is linearly independent with $\vec{Y}_{1,l'}(t),\dots,\vec{Y}_{1,r}(t)$. Indeed, such an element would also be a horizontal element at $a_1$ which would contradict the optimality of the chosen basis. 
 
 This shows that the process is compatible, that is, we may continue it by choosing some $j_0$ for which we have constructed part of the optimal basis in the first step, and complete this part to a full optimal basis. Then, we continue by repeating the first step at $a_{j_0}$ instead of $a_1$ and leaving all the parts of optimal basis constructed in the first step intact. 
 
 Finally, we do the same for the remaining points (that is the ones for which we did not construct part of the optimal basis in the steps above).  
\end{proof}

\section{Optimal bases}\hfill\\

%
%

\subsection{Trivial differential module}\hfill\\

We recall that we fix $\{e_1,\dots,e_\fr\}$ a basis for $M$ and $\{e_1,\dots,t^{d-1}e_\fr\}$ a basis for $M_\vphi$. Hence all elements of the corresponding modules will be identified with the column vectors of their coordinates with respect to these bases. 

Further, let $0<r_1<\dots<r_n\leq 1$ be the branching radii of branching points of $\T_{\vphi,b}$ (see Definition \ref{defn: branching radii}).

\begin{defn}
Let $\cU\subset \sD_t^\pm$ be an open disc. We denote by $v_\cU$ the column vector of dimension $d$ whose $i$-th component is 1 if $a_i\in \cU(k)$, and 0 otherwise.
\end{defn}

\begin{rmk}\label{rmk: bigger disc}
For each open disc $\cU\subseteq \sD_t^\pm$ one can find another open disc $\cU'\subseteq \sD_t^\pm$ such that either $\cU'=\sD_t^-$, $\eta_{0,1}$ is not branching and $v_{\cU}=v_{\cU'}$, either $\cU'$ is a branch at some $\eta\in \cB_{\vphi,b}$ and  $v_{\cU}=v_{\cU'}$. 
\end{rmk}

\begin{lemma}\label{lem: radius exactly}
Let $\eta\in \cB_{\vphi,b}(r_i)$ and let $\cU\in \cB r_\eta$. 

 Then, $\cR_b(\bV(s)\, v_\cU)=r_i$. Furthermore, $\bV(s)\, [1,\dots,1]^T=E_1:=[1,0,\dots,0]^T$.
\end{lemma}
\begin{proof}
 Let us prove that $E_1=\bV(s)\, [1,\dots,1]^T$. For this, let us put 
 $$
 \bV(s)\, [1,\dots,1]^T=:[g_1(s),\dots,g_d(s)]^T.
 $$
 Then, 
 $$
 \bU(s)\, [g_1(s),\dots,g_d(s)]^T=[1,\dots,1]^T,
 $$
 which implies that the polynomial $(g_1(s)-1)+g_2(s)\cdot X+\dots+g_d(s)\cdot X^{d-1}$ has zeroes in $u_{a_1}(s),$ $\dots,u_{a_d}(s)$. Since it is of degree $d-1$, it must be identically zero, hence $g_1(s)=1$ and $g_2(s)=\dots=g_d(s)=0$.
 
 By our choice of $\cU$, we have that $\vphi(\cU)=\sD_s(b,r_i^-)$ and if we put $\vphi^{-1}\big(\sD_s(b,r_i^-)\big)=\cup_{j=1}^n\sD_t(a_j',q_j^-)$, then $\cU$ is equal to one of the components $\sD_t(a_j',q_j^-)$. 
 
 Now, let $A_\cU\in\bigoplus_{j=1}^l\cO_t(a'_j,q_j^-)$ such that its restriction to $\cU$ is 1, while its restriction to any other connected component of $\vphi^{-1}(\sD_s(b,r_i^-))$ is 0. Note that $A_\cU$ is a horizontal element of $T\otimes \oplus_{j=1}^l\cO_t(a_j',q_j^-)$ hence gives rise to a horizontal element of $(M_\vphi)_{b,r_i^-}$. Its image by the map \eqref{eq: res map} can be identified with $v_\cU$ hence the corresponding horizontal element in the direct image is precisely $\bV(s)\, v_\cU$ as Lemma \ref{lem: new coord} suggests. We conclude that $\cR_b(\bV(s)\, v_\cU)\geq r_i$. If $r_i=1$ we are done, so suppose that $r_i<1$, in which case also $q_j<1$. 
 
 We note that for any $R>q_j$, $R\in (0,1)$ the restriction of $1\cdot e$ to $\sD_t(a_j',R^-)$ is a basis of horizontal elements of $M_{a_j',R^-}$ and it is the only horizontal element whose restriction to $\sD_t(a_j',q_j^-)$ is again $1e$. However, since $\eta$ is a branching point and $\eta\in \sD_t(a_j',R^-)$, there is always some $\sD_t(a_l',q_l^-)$ which is a connected component of $\vphi^{-1}(\vphi(U))$ different from $\sD_t(a_j',q_j^-)$ and such that  $\sD_t(a_l',q_l^-)\subset \sD_t(a_j',R^-)$. In particular the restriction of $1\cdot e$ to this disc is different from 0. Lemma \ref{lem: radius crit} then implies that $\cR_b(\bV(s)\, v_\cU)= r_i$.
 \end{proof}

\begin{thm}\label{thm: opt basis trivial}
 For each $i=1,\dots,n$ and $\eta\in\cB_{\vphi,b}(r_i)$, let $\sU_\eta$ denote any set of $\delta(\eta)-1$ branches at $\eta$. Let further 
 $$
 \Pi_i:=\{\bV(s)\, v_\cU\mid \cU\in \sU_{\eta},\quad \eta \in \cB_{\vphi,b}(r_i)\}.
 $$
Then,
\begin{enumerate}
 \item Every function in $\Pi_i$ has $r_i$ as radius of convergence at $b$.
 \item $\bigcup\limits_{i=1}^n\Pi_i\cup \{E_1\}$, where $E_1:=[1,0,\dots,0]^T=\bV(s)\, [1,\dots,1]^T$, 
 is an optimal basis for $(T_\vphi,D_s)$ at $b$.
\end{enumerate}
\end{thm}

 We will need the following lemma.
 
 \begin{lemma}\label{lem: help}
  Let $\sD\subseteq \sD_t^\pm$ be an open disc. Then, $v_{\sD}$ is a linear combination of vectors $v_{\cU}$ where $\cU\in \sU_{\eta},\quad \eta \in \cB_{\vphi,b}(r_i)$, $i=1,\dots,n$ and vector $[1,\dots,1]^T$. Furthermore, in this linear combination only vectors $v_\cU$ with radius of disc $\vphi(\cU)$ bigger than or equal to the radius of disc $\vphi(\sD)$ can have non-zero coefficients and at least one such vector appears with nonzero coefficient.  
 \end{lemma}
 \begin{proof}
 By Remark \ref{rmk: bigger disc} we may assume that $\sD$ is either an open disc $\sD_t$ or a branch at some point $\eta\in \cB_{\vphi,b}$, hence the radius of $\vphi(\sD)$ is in $\{r_1,\dots,r_n,1\}$. We argue inductively on the size of this radius. Suppose that the radius of $\vphi(\sD)$ is 1. We distinguish two cases: 
 
 \emph{Case 1.} Our module $T$ is defined over the open unit disc $\sD_t^-$, hence $\sD=\sD_t^-$. Then, $v_{\sD}=[1,\dots,1]$.

\emph{Case 2.} Our module $T$ is defined over the closed unit disc $\sD_t^+$. In this case, $\sD$ is either one of the branches at $\eta_G$ either it is not and in the latter case $v_{\sD}$ is the zero vector. If $\sD$ is equal to one of the branches at $\eta_G$, then if $\sD$ is equal to one of the discs in $\sU_{\eta_G}$ we are done. If not, then $\sD_t$ is equal to the remaining branch at $\eta_G$ which is not in $\sU_{\eta_G}$, hence we may write
 $$
v_{\sD}=[1,\dots,1]^T-\sum_{\cU\in\sU_{\eta_G}}v_{\cU}. 
 $$
 
 Suppose now that the radius of $\vphi(\sD)$ is strictly smaller than 1, and equal to some $r_i$, $i=1,\dots,n$, and that Lemma is true for all the discs $\sD''\subseteq \sD_t^\pm$ for which $\vphi(\sD'')>r_i$ (hence equal to some $r_{i+1},\dots,r_n$ or 1).   We may assume that $\sD$ is a branch at some branching point $\eta\in\cB{\vphi,b}$ (otherwise $v_\sD$ is a zero vector). Then, if $\sD$ is one of the branches in $\sU_\eta$ we are done. If not, then $\sD$ is the remaining branch at $\eta$ which is not in $\sU_\eta$. Let $\sD'$ be any open disc in $\sD_t^\pm$ which contains $\eta$ and no other branches except those that are attached to $\eta$. Let further $\sD''$ be any branch so that $v_{\sD'}=v_{\sD''}$, or if no such branch exists, let $\sD''$ be disc $\sD_t^-$. Then, we may write   
 $$
v_\sD=v_{\sD''}-\sum_{\cU\in \sU_\eta}v_\cU. 
 $$
 However, the radius of $\vphi(\sD'')$ is strictly bigger than $r_i$ hence the inductive hypothesis applies, which finishes the proof. 
 \end{proof}

\begin{proof}[Proof of Theorem \ref{thm: opt basis trivial}]

 Part \emph{(1)} follows from Lemma \ref{lem: radius exactly}. 
 
  As for the part \emph{(2)}, we will use Lemma \ref{lem: optimal criterion} \emph{(3)}. We note that Lemma \ref{lem: sum of branches} implies that the cardinality of $\bigcup\limits_{i=1}^n\Pi_i\cup \{e_1\}$ is $d$.
 
 Let $\vec{f}(s)$ be a horizontal element of $T_\vphi$. Suppose that $\cR_b(\vec{f}(s))=r$ and let us put $\vphi^{-1}(\sD_s(b,r^-))=\cup_{i=1}^l\sD_t(a_i',r_i^-)$, the union being disjoint. By Remark \ref{rmk: bijection horizontal} and Corollary \ref{cor: iso diff} there are horizontal elements $m_i\in M^{D_t}_{a_i',r_i^-}$ such that $(\oplus_{i=1}^lm_i)_\vphi=\vec{f}(s)$. For $i'=1,\dots,l$, let us denote by $\wtilde{m}_{i'}$ the element in $\oplus_{i=1}^lM_{a_i',r_i^-}$ whose image under projection $\oplus_{i=1}^lM_{a_i',r_i^-}\to M_{a_j',r_j^-}$ is $m_{i'}$ if $j=i'$ and 0 otherwise. Thus, we may also write $(\sum_{i=1}^l\wtilde{m}_i)_\vphi=\sum_{i=1}^l(\wtilde{m}_i)_\vphi=\vec{f}(s)$. 

 However, each element $m_i$ is of the form $\alpha_i\cdot e_{|\sD_t(a_i',r_i^-)}$, for some $\alpha_i\in k$. Hence, by paragraph \ref{ssec: detail coordinates} and Lemma \ref{lem: new coord} we have that $(\wtilde{m}_i)_\vphi=\bV(s)\,v_{\sD_t(a_i',r_i^-)}$. 
 
 Then, Lemma \ref{lem: help} insures that $\vec{f}(s)$ is a linear combination of the asserted vectors, and together with Lemma \ref{lem: radius crit} implies that all the conditions in Lemma \ref{lem: optimal criterion} \emph{(3)} are satisfied. The theorem follows.
\end{proof}

%
%
%

\subsection{General case}\hfill\\

\subsubsection{} Let $\vec{Y}_{a_i,1}(t),\dots, \vec{Y}_{a_i,\fr}(t)$ be an optimal basis for $(M,D_t)$ at $a_i$, $i=1,\dots,d$. Suppose further that they are linked. For each $\vec{Y}_{a_i,j}(t)$, let us put $U_{i,j}:=\sD_t(a_i,r_{i,j}^-)$, where $r_{i,j}$ is the radius of convergence of $\vec{Y}_{a_i,j}(t)$ at $a_i$.

\begin{defn}
 We say that $(\vec{Y}_{a_i,j}(t),U_{i,j})$ is a \emph{fundamental pair} and we denote the set of fundamental pairs by $\sP$.
\end{defn}

\begin{rmk}
 We note that if $(M,D_t)=(T,D_t)$ is a trivial differential module over an open unit disc $\sD_t^-$ with basis of horizontal elements given by $1\cdot e$, then there is just one fundamental pair, namely $(1\cdot e,\sD_t^-)$. 
\end{rmk}
\begin{defn}

Let $\cP=(\vec{Y}(t),U)\in \sP$. We put $\cB_{\vphi,b}(\cP):=U\cap\cB_{\vphi,b}$. For each $\eta\in \cB_{\vphi,b}$, let $\sU_\eta$ denote any set of $\delta(\eta)-1$ branches at $\eta$. We set $\sU_\cP$ to denote the set of all branches chosen in this way together with $U$.

 Let $\eta\in \cB_{\vphi,b}(\cP)$ and let $\cU$ be any branch at $\eta$. We define two vectors of analytic functions
\begin{itemize}
 \item Vector $v_{\cP,\cU,t}$ whose $\fr$ entries in positions $(i-1)\cdot \fr+1,\dots,i\cdot \fr$ are $\fr$ entries of $\vec{Y}(t)$ if $a_i\in \cU$, $i=1,\dots,d$ and are $0$ otherwise. 
 \item Vector $v_{\cP,\cU,s}$ whose $\fr$ entries in positions $(i-1)\cdot \fr+1,\dots,i\cdot \fr$ are $\fr$ entries of $\vec{Y}(u_{a_i}(s))$ if $a_i\in \cU$, $i=1,\dots,d$, and are $0$ otherwise. 
\end{itemize}
\end{defn}

For the convenience, we sum up the results of Paragraph \ref{ssec: detail coordinates} and Lemma \ref{lem: new coord} in the following:

\begin{lemma}\label{lem: sum up}
 Let $\cP=(\vec{Y}(t),U)\in \sP$ and let $\vphi^{-1}\big(\vphi(U)\big)=\cup_{i=1}^l\sD_t(a_i',r_i^-)$, the latter union being disjoint, and let $i'$ be such that $U=\sD_t(a_{i'}',r_{i'}^-)$. Further, let $\cU\in\sU_\sP$ and let $m_{\cP,\cU}$ be the horizontal element of $M\otimes \cO_t(a_{i'},r_{i'}^-)$ that corresponds to the restriction of $\vec{Y}(t)$ to $\cU$ and $\wtilde{m}_{\cP,\cU}$ be the horizontal element of $\oplus_{i=1}^lM\otimes \cO_t(a_i',r_i^-)$ whose image under projection $\oplus_{i=1}^lM\otimes \cO_t(a_i',r_i^-)\to M\otimes \cO_t(a_j',r_j^-)$ is $m_{\cP,\cU}$ if $j=i'$ and $0$ otherwise.
 
 Then, the image of $\wtilde{m}_{\cP,\cU}$ (with respect to obvious basis) under the map \eqref{eq: res map} can be identified with $v_{\cP,\cU,t}$ while $(\wtilde{m}_{\cP,\cU})_\vphi=\bV_\fr(s)\,v_{\cP,\cU,s}$.
\end{lemma}

Then, if an $r\in (0,1)$ is small enough, so that $\vphi^{-1}(\sD_s(b,r^-))=\cup_{i=1}^d\sD_s(a_i,r_i^-)$, $v_{\cP,\cU,t}$ can be identified with an element of the differential module $M\otimes \bigoplus_{i=1}^d\cO_t(a_i,r_i^-)$. In fact, by construction, $v_{\cP,\cU,t}$ is horizontal element. 

On the other side, by construction from Paragraph \ref{ssec: detail coordinates} and Lemma \ref{lem: new coord}, we see that $(v_{\cP,\cU,t})_\vphi=\bV_\fr(s)\,v_{\cP,\cU,s}$ so that $\bV_\fr(s)\,v_{\cP,\cU,s}$ is a horizontal element of the direct image differential module $M_\vphi\otimes \cO_s(b,r^-)$.

\begin{rmk}
 Once again, we note that in case of a trivial differential module $(T,D_t)$, the set $\sP$ just constructed coincides with the set of elements $(1e,\cU)$, where $\cU$ are as in Theorem \ref{thm: opt basis trivial}.  
\end{rmk}

Finally, we reach the main result of this article. 

\begin{thm}\label{thm: main}
 Keeping the setting as before, we have:
 \begin{enumerate}
  \item The radius of convergence at $b$ of horizontal element $\bV_\fr(s)\, v_{\cP,\cU,s}$ is equal to the radius of disc $\vphi(\cU)$.
  \item An optimal basis of $(M_\vphi,D_s)$ at $b$ is given by $\bV_\fr(s)\, v_{\cP,\cU,s}$, where $\cU\in\sU_\cP$, $\cP\in \sP$.
 \end{enumerate}

\end{thm}

Before going to the proof, we will need the following generalization of Lemma \ref{lem: help}.

\begin{lemma}\label{lem: help 2}
 Let $\cV$ be any branch at some $\eta\in\cB_{\vphi,b}(\cP)$. Then, $v_{\cP,\cV,t}$ is a linear combination of vectors of the form $v_{\cP,\cU,t}$, where $\cU\in\sU_\cP$ and moreover only vectors with the radius of $\vphi(\cU)$ bigger than or equal to the radius of $\vphi(\cV)$ have nonzero coefficients with at least one such vector appearing with nonzero coefficient. 
\end{lemma}
\begin{proof}
 The proof is similar to that of Lemma \ref{lem: help}, so we just sketch it.  Namely, if $\cV\in\sU_\cP$ there is nothing to prove. If $\cV\notin \sU_{\cP}$, there is an $\eta'\in \cB_{\vphi,b}(\cP)$ and a branch $\cV'$ at $\eta'$ so that $\cV'$ contains no point of $\cB_{\vphi,b}(\cP)$ whose radius is bigger than $\eta$. Then, it is not difficult to see that 
 $$
 v_{\cP,\cV,t}=v_{\cP,\cV',t}-\sum_{\cU\in\sU_\eta}v_{\cP,\cU,t}.
 $$
 Applying the same reasoning now to $v_{\cP,\cV',t}$ will lead us eventually to the case where $\cV'\in \sU_\cP$ which finishes the proof. 
\end{proof}

\begin{cor}
 The set $\bV_\fr(s)\, v_{\cP,\cU,s}$, $\cU\in \sU_\cP$, $\cP\in \sP$ is a basis of horizontal elements for $M_\vphi$ at $b$.
\end{cor}
\begin{proof}
The previous lemma shows that the set generates the space of horizontal elements. Indeed, for $\vec{Y}_{a_i,j}(t)$, let $\cU_{i,j}$ be the open disc so that $(\vec{Y}_{a_i,j}(t),\cU_{i,j})\in \sP$. Further, for each $i$, let $\cV_i$ denote the branch in $\sD_t^\pm$ that contains $a_i$ and no other point in $\vphi^{-1}(b)$. Then, Theorem \ref{thm: fund sol} shows that a basis of horizontal elements for $M_\vphi$ at $b$ is given by functions $\bV_\fr(s)\,v_{(\vec{Y}_{a_i,j}(t),\cU_{i,j}),\cV_i,s}$, $i=1,\dots, d$, $j=1,\dots,\fr$.  The previous lemma then shows that these functions can be expressed as linear combinations of the required ones.

It remains to prove that the asserted set has $\fr\cdot d$ elements. If we start with $\fr\cdot d$ pairs $(i,j)$, $i=1,\dots,\fr$, $j=1,\dots,d$, and introduce on them an equivalence relation $\sim$ with $(i,j) \sim (i',j')$ if and only if $j=j'$ and $a_{i'}\in \cU_{i,j}$, then we see that the number of classes is equal to the number of different fundamental pairs $(\vec{Y}_{a_i,j},\cU_{i,j})$, and that the number of elements in a class $[(i,j)]$ is equal to $\#\big(\cU_{i,j}\cap \vphi^{-1}(b)\big)$ (recall that the chosen basis are linked). That is
$$
\sum_{(\vec{Y}(t),\cU)\in \sP}\#\big(\cU\cap \vphi^{-1}(b)\big)=\fr\cdot d
$$

However, Lemma \ref{lem: sum of branches} implies that  $\#\big(\cU\cap \vphi^{-1}(b)\big)=\#\sU_{(\vec{Y}(t),\cU)}$, hence the Corollary. 
\end{proof}

\begin{proof}[Proof of Theorem \ref{thm: main}]
\emph{(1)} Let $r\in (0,1]$ such that $\vphi(\cU)=\sD_s(b,r^-)$ and put $\vphi^{-1}\big(\sD_s(b,r^-)\big)=\cup_{i=1}^l\sD_t(a_i',r_i^-)$. Let further $i'$ and $m_{\cP,\cU}$ be as in Lemma \ref{lem: sum up}. Then, by Lemma \ref{lem: radius crit} it is enough to prove that for any $R>r_{i'}$, there is no horizontal element $m'$ of $M\otimes \cO_t(a_{i'},r_{i'}^-)$ whose restriction to $\cU$ is $m_{\cP,\cU}$ and for this, we may assume that $r_{i'}<1$.

If $\cU=U$ the claim is clear, as in this case $r_{i'}$ is the radius of convergence of $m_{\cP,\cU}$ hence it cannot be a restriction of an element having a bigger radius of convergence. If $\cU\neq U$, then $\cU$ is a branch at some point $\eta\in \cB_{\vphi,b}(\cP)$. Let $\cV$ be any other branch at $\eta$. Then, if $m'$ is the asserted horizontal element with radius of convergence $R$, its restriction to $\cV$ must be 0, by the construction of $v_{\cP,\cU,t}$ hence it is zero everywhere which is a contradiction. Hence, $\cR_b(\bV_r(s)\,v_{\cP,\cU,s})=r$.

\emph{(2)} To prove that the basis is optimal, we will use criterion \emph{(3)} in Lemma \ref{lem: optimal criterion}. 

Let $\vec{f}(s)$ be a horizontal element of $(M_\vphi,D_s)$ at $b$ and let $\cR_b\big(\vec{f}(s)\big)=r$. As usual, let us put $\vphi^{-1}\big(\sD_s(b,r^-)\big)=\cup_{i=1}^l\sD_t(a_i',r_i^-)$ and let $\vec{g}_i(t)$ be the horizontal element of $M\otimes \cO_t(a_i',r_i^-)$ such that $\big(\oplus_{i=1}^l\vec{g}_i(t)\big)_\vphi=\vec{f}(s)$. 

Then, each $\vec{g}_i(t)$ is a horizontal element of $M$ at $a_i'$  of radius of convergence greater than or equal to $r_i$ and hence can be written as a linear combination of the form 
$$
\vec{g}_i(t)=\sum_{j=1}^\fr\alpha_{i,j}\cdot\vec{Y}_{a_i',j}(t),
$$
where by Lemma \ref{lem: optimal criterion} only those $j$ for which $\vec{Y}_{a_i',j}(t)$ has radius of convergence greater than or equal to $r_i$ may have $\alpha_{i,j}\neq 0$. 
 Let us put for convenience $I$ to denote a set of those $i$ for which there is some $j$ with $\alpha_{i,j}\neq 0$. Further, for $i\in I$, let $I(i)$ denote the set of those $j$ for which $\alpha_{i,j}\neq 0$ and let $j(i)$ denote minimum of $I(i)$.

We will argue inductively on $r$ (since the radii of convergence of all horizontal elements form a discrete set). 

Suppose that $r=1$. Then, each $r_i=1$ and for $i\in I$ and $j\in I(i)$ we have $\big(\vec{Y}_{a_i',j}(t),\sD_t(a_i',1^-)\big)\in\sP$ and  
$$
\vec{f}(s)=\sum_{i\in I}\sum_{j\in I(i)}\alpha_{i,j}\cdot\bV_\fr(s)\,v_{\big(\vec{Y}_{a_i',j}(s),\sD_t(a_i',1^-)\big),\sD_t(a_i',1^-),s},
$$
which affirms Lemma \ref{lem: optimal criterion} \emph{(3)}.

Suppose now that $r<1$ and that the statement of the theorem holds for every horizontal element of $M_\vphi$ which has radius of convergence bigger than $r$. For $i=1,\dots,l$, let $I'(i)\subseteq\{1,\dots,\fr\}$ contain those $j$ for which the function $\alpha_{i,j}\cdot\vec{Y}_{a_i',j}(t)$ has radius of convergence exactly $r_i$. Furthermore, let $I''(i)\subseteq \{1,\dots,\fr\}$ contain those $j$ for which the function $\alpha_{i,j}\cdot\vec{Y}_{a_i',j}(t)$ has radius of convergence $r_{i,j}>r_i$ and $\sD_t(a_i',r_i^-)$ is a branch at some $\eta\in \cB_{\vphi,b}\big((\vec{Y}_{a_i',j},\sD_t(a_i',r_{i,j}^-))\big)$. Moreover, in this case we also ask that there is a branch $\cV$ at $\eta$ that contains some $a_o'\in\{a_1',\dots,a_l'\}$ and 
$$
\alpha_{o,j}\cdot\vec{Y}_{a_o',j}(t)=\alpha_{o,j}\cdot\vec{Y}_{a_i',j}(t)\neq \alpha_{i,j}\cdot\vec{Y}_{a_i',j}.
$$
(Note that we have $\vec{Y}_{a_o',j}(t)=\vec{Y}_{a_i',j}(t)$ since the bases are linked.) Finally we put $I_0(i):=I'(i)\cup I''(i)$ and set $I_0$ to denote the set of those $i=1,\dots,l$, for which $I_0(i)\neq \emptyset$.

\emph{Claim 1.
The radius of convergence of 
$$
\vec{Z}_1(s):=\left(\bigoplus_{i=1}^l\sum_{j\in I_0(i)}\alpha_{i,j}\cdot\vec{Y}_{a_i',j}(t)\right)_\vphi
$$
at $b$ is equal to $r$.
}
\begin{proof}[Proof of Claim 1.]
This amounts to spelling down the discussion before Lemma \ref{lem: radius crit}.

Suppose that this radius is $R>r$ and let us put $\vphi^{-1}\big(\sD_s(b,R^-)\big)=\cup_{i=1}^{l'}\sD_t(a_i'',R_i^-)$, where the union is disjoint and where as assume, as we can, that $\{a_1'',\dots,a_{l'}''\}\subseteq \{a_1',\dots,a_l'\}$. Let $i\in\{1,\dots,l'\}$ and $i_0\in\{1,\dots,l\}$ with $a_{i_0}'=a_i''$ and with $I_0(i_0)\neq \emptyset$. Further, for each $i=1,\dots,l'$, let 
$$
\vec{G}_i(t):=\sum_{j=1}^\fr\beta_{i,j}\cdot \vec{Y}_{a_i'',j}(t)
$$
be the horizontal element of $M_{a_i'',R_i^-}$ so that we have $\left(\oplus_{i=1}^{l'}\vec{G}_i(t)\right)_\vphi=\vec{Z}_1(s)$. We also note that  each function $\beta_{i,j}\cdot\vec{Y}_{a_i'',j}(t)$ has radius of convergence at least $R_i$ (Lemma \ref{lem: optimal criterion}).

On the other side, the restriction of $\vec{G}_i(t)$ to $\sD_t(a_{i_0}',r_{i_0}^-)$ is $\vec{g}_{i_0}(t)$, which gives us $\beta_{i,j}=\alpha_{i_0,j}$ and $\vec{Y}_{a_i'',j}=\vec{Y}_{a_{i_0}',j}(t)$. This already implies that $I'(i)$ is empty, hence $I_0(i)=I''(i)$. Let $j\in I''(i)$ and let $a_o'$ and $\cV$ be as in the definition of $I''(i)$. Then, $\cV=\sD_t(a_o',r_i^-)\subseteq \sD_t(a_i'',R_i^-)$ and the restriction of $\vec{G}_i(t)$ to $\cV$ is $\vec{g}_{o}(t)$. However, this is impossible since $\beta_{i,j}\vec{Y}_{a_i'',j}(t)=\alpha_{i_0,j}\cdot \vec{Y}_{a_{i_0}',j}(t)\neq \alpha_{o,j}\cdot\vec{Y}_{a_o,j}(t)$. Hence, $I_0(i)$ must be empty.
\end{proof}

\emph{Claim 2. The radius of convergence of 
$$
\vec{Z}_2(s):=\vec{f}(s)-\vec{Z}_1(s)=\left(\bigoplus_{i=1}^l\sum_{\substack{j=1\\j\notin I_0(i)}}^\fr\alpha_{i,j}\cdot \vec{Y}_{a_i',j}(t)\right)_\vphi
$$
at $b$ is bigger than $r$.
}
\begin{proof}[Proof of Claim 2.]
For a fixed $i=1,\dots,d$ and $j\in\{1,\dots,\fr\}\setminus I_0(i)$ we note that each function $\alpha_{i,j}\cdot\vec{Y}_{a_i',j}(t)$ has radius of convergence bigger than $r_i$. Let $r_i'$ denote the smallest among these radii (hence $r'_i>r_i$). 

Let $\eta\in \cB_{\vphi,b}$ such that $\sD_t(a_i',r_i^-)$ is a branch at $\eta$ and let $\cV_{i,1}:=\sD_t(a_{o_1}',r_i^-),\dots,\cV_{i,m}:=\sD_t(a_{o_m}',r_i^-)$ be all the other branches at $\eta$ (with $o_1,\dots,o_m\in \{1,\dots,l\}$). Let $\cV_i$ be an open disc of radius at most $r_i'$ that contains $\eta$ and such that $\cV_i\setminus\{\eta\}$ has no other branch as a connected component except for $\cV_{i,1},\dots,\cV_{i,m}$. Let $r_i''$ be the radius of $\vphi(\cV_i)$ and finally put $R:=\min\{r_1'',\dots,r_l''\}$.

Let $\vphi^{-1}\big(\sD_s(b,R^-)\big)=\cup_{i=1}^{l'}\sD_t(a_i'',R_i^-)$, union being disjoint and with the usual assumption that $\{a_1'',\dots,a_{l'}''\}\subseteq \{a_1',\dots,a_l'\}$. For $i=1,\dots,l'$ let $o(i)$ be any index so that we have $a_{i}''=a_{o(i)}'$. Then, by previous constructions we have that the function 
$$
\vec{H}_{i}(t):=\sum_{\substack{j=1\\j\notin I_0(i)}}^\fr\alpha_{o(i),j}\cdot \vec{Y}_{a_{o(i)}',j}(t)
$$
restricts to 
$$
\sum_{\substack{j=1\\j\notin I_0(n)}}^\fr\alpha_{n,j}\cdot\vec{Y}_{a_n',j}(t)
$$
for every $n$ for which $\sD_t(a_{n}',r_n^-)\subset \sD_t(a_i'',R_i^-)$. Lemma \ref{lem: radius crit} then implies the claim.
\end{proof}

We continue the proof of the theorem. Two claims imply that $I_0\neq \emptyset$ because of our assumption on the radius of convergence of $\vec{f}(s)$. Furthermore, for every $i\in I_0$ and $j\in I'(i)$, we have  $\big(\vec{Y}_{a_i',j}(t),\sD_t(a_i',r_i^-)\big)\in\sP$ and by part \emph{(1)} of the theorem, the radius of convergence of function $\bV_\fr(s)\,v_{\big(\vec{Y}_{a_i',j}(t),\sD_t(a_i',r_i^-)\big),\sD_t(a_i',r_i^-),s}$ is equal to $r$. On the other side, if $j\in I''(i)$, then Lemma \ref{lem: help 2} implies that $\bV_\fr(s)\,v_{\big(\vec{Y}_{a_i',j}(t),\sD_t(a_i',r_{i,j}^-)\big),\sD_t(a_i',r_i^-),s}$ is a linear combination of the functions $\bV_\fr(s)\,v_{\cP,\cU,s}$, where $\cU\in \sU_\cP$, $\cP\in \sP$, and where only those with radius of convergence bigger than or equal to $r$ may have nonzero coefficients (with at least one such function appearing with nonzero coefficient). 

In other words, $\vec{Z}_1(s)$ can be written as indicated by criterion \emph{(3)} in Lemma \ref{lem: optimal criterion} and since by inductive hypothesis the same is true for $\vec{Z}_2(s)$ as well, it must also be true for $\vec{f}(s)$. The theorem follows.
\end{proof}


%
%

\begin{rmk}
 Theorem \ref{thm: main} should be compared to \cite[Theorem 3.6.]{BojPoi} and can be seen as a refinement of the latter result.
\end{rmk}

\begin{rmk}
Instead of horizontal elements, one may look for solutions of a differential module $(M,D_t)$ at $a\in\sD_t^\pm$. Recall that these are the horizontal morphisms of differential modules $(M,D_t)\to (k[[t-a]],d_t)$. Then, once we fix a basis for $M$, a solution corresponds to a vector of analytic functions in $k[[t-a]]^\fr$ (the entries are the images of the elements of the basis). Then, one may ask what are the analogues of Theorems \ref{thm: fund sol}, \ref{thm: main} and \ref{thm: opt basis trivial} in this setting?

To the best of our knowledge, the only result in this direction is an analogue of Theorem \ref{thm: fund sol} for the case of a trivial differential module presented in \cite{BaldaMilano}. We will address this question in a subsequent article.
\end{rmk}

\subsection{Examples}\hfill\\

Let $\vphi:\sD_t^-\to \sD_s^-$ be defined by $s=f(t):=\sum_{i=1}^p{p \choose i}\cdot t^i$, where $p$ is the residual characteristic of $k$. Then, $\vphi$ is a finite, \'etale morphism of degree $p$ (off-centered Frobenius map). 

Let us put $f_p(s):=1+\sum_{j=1}^\infty{1/p\choose j}\cdot s^j$. Then, as one can formally check we have $f_p(s)^p=1+s$ and it is classical fact that $\cR_0\big(f
_p(s)\big)=|p|^{\frac{p}{p-1}}$ as can be checked by studying valuation polygon of $f_p(s)$ (see for example \cite[Chapter II]{DGS}). 

Furthermore, we have that the preimages of $0$ by $\vphi$ are given by $a_i:=-1+\zeta_p^i$, $i=1,\dots,p$, where $\zeta_p$ is a primitive $p$-root of 1.

The associated polynomial giving the morphism $\vphi$ is $P(s,X):=\sum_{i=1}^p{p \choose i}\cdot X^i-s$ and $p$ solutions of $P(s,X)=0$ are given by 
$$
u_{a_i}(s):=-1+\zeta_p^i\cdot f_p(s)=a_i+\sum_{j=1}^\infty\zeta_p^i\,{1/p \choose j}\cdot s^j.
$$ 
Further, from relation $s+1=(t+1)^p=f(t)+1$, one obtains 
\begin{equation}\label{eq: derivative}
\frac{1}{f'(t)}=\frac{1}{p\cdot(s+1)}+\frac{1}{p\cdot(s+1)}t.
\end{equation}
\subsubsection{} Suppose that $p=2$. Then, $\T_{\vphi,0}$ is given by two segments emanating from points $0$ and $-2$ (the preimages of $0$) that meet at point $\eta_{0,|2|}=\eta_{-2,|2|}$. This point is the only branching point and it has two branches $\sD_t(0,|2|^-)$ and $\sD_t(-2,|2|^-)$, both of which are sent by $\vphi$ to $\sD_s(0,|2|^2)$.

If we consider trivial differential module $(T,D_t)$ and its direct image by $\vphi$, then by using \eqref{eq: derivative} we obtain that the the system of differential equations associated to $(T_\vphi,D_s)$ it (see \eqref{eq: assign system} and discussion before it) is given by 
$$
\frac{d}{ds}\vec{Y}(s)=-\frac{1}{2}\cdot \begin{bmatrix}
                              0 & \frac{1}{s+1}\\
                              0 & \frac{1}{s+1}
                             \end{bmatrix}\,\vec{Y}(s).
$$
Of course, the previous system can be solved directly and with a little effort an optimal basis can be found. However, we can calculate 
$$
\bV(s)=\frac{1}{2\cdot f_2(s)}\cdot\begin{bmatrix}
       -1+f_2(s) & 1+f_2(s)\\
       -1 & 1
       \end{bmatrix}
$$
and by Theorem \ref{thm: opt basis trivial} obtain an optimal basis for $(T,D_t)_\vphi$ at $0$ in the form
$$
\Bigg\{\bV(s)\,\begin{bmatrix}
                1\\
                1
               \end{bmatrix},\bV(s)\,\begin{bmatrix}
                0\\
                1
               \end{bmatrix}
\Bigg\}=\Bigg\{\begin{bmatrix}
                1\\
                0
               \end{bmatrix},\begin{bmatrix}
                \frac{1}{2\cdot f_2(s)}+\frac{1}{2}\\
                \frac{1}{2\cdot f_2(s)}
               \end{bmatrix}
\Bigg\},
$$
where the second horizontal element has $|2|^2$ as radius of convergence. 

If instead we consider differential module $(M=\cO_t^-\, e,D_t)$, given by $D_t(e)=-e$, we see that for every $a\in \sD_t^-(k)$, a horizontal element at $a$ is given by $\exp(t-a)\cdot e$ with radius of convergence $|2|$. Then, we have two fundamental pairs $\cP_1:=\big(\exp(t+2),\sD_t(-2,|2|^-)\big)$ and $\cP_2:=\big(\exp(t),\sD_t(0,|2|^-)\big)$, so Theorem \ref{thm: main} gives us that an optimal basis for $(M_\vphi,D_s)$ at $0$ is given by 
\begin{align*}
\Bigg\{v_{\cP_1,\sD_t(-2,|2|^-),s}, &v_{\cP_1,\sD_t(0,|2|^-),s}\Bigg\}=\\ &\Bigg\{\begin{bmatrix}
                              \frac{-1+f_2(s)}{2\cdot f_2(s)}\cdot \exp(1-f_2(s))\\
                             \frac{-1}{2\cdot f_2(s)}\cdot \exp(1-f_2(s))
                             \end{bmatrix},\begin{bmatrix}
                             \frac{1+f_2(s)}{2\cdot f_2(s)}\cdot \exp(-1+f_2(s))\\
                             \frac{1}{2\cdot f_2(s)}\cdot \exp(-1+f_2(s))
                             \end{bmatrix}
\Bigg\},
\end{align*}
both of which have the same radius of convergence, namely $|2|^2$.

On the other side, since we have 
$$
\frac{1}{f'(t)}\cdot D_t(e)=-\frac{1}{2\cdot(s+1)}\cdot e-\frac{1}{2\cdot(s+1)}\cdot t\cdot e
$$
and 
$$
\frac{1}{f'(t)}\cdot D_t(t\cdot e)=\frac{1-s}{2\cdot(s+1)}\cdot e+ \frac{1}{s+1}\cdot t\cdot e,
$$
the associated differential system to $(M,D_t)_\vphi$ is given by 
$$
\frac{d}{ds}\vec{Y}(s)=\frac{1}{2}\begin{bmatrix}
                             \frac{1}{s+1} & \frac{s-1}{s+1}\\
                             \frac{1}{s+1} & -\frac{1}{s+1}
                             \end{bmatrix}\vec{Y}(s).
$$
and one may solve the previous system directly.

\subsubsection{} Suppose now that $p=3$. Then, preimages of $0$ are given by $-1+\zeta_3^i$, $i=1,2,3$ and $\T_{\vphi,0}$ is given by three segments emanating from these points and which meet at $\eta_{0,|3|^{\frac{1}{2}}}$. It is the only branching point and there are three branches $\sD_t(-1+\zeta_3^i,(|3|^{\frac{1}{2}})^-)$, $i=1,2,3$ all of which are sent to $\sD_t(0,(|3|^{\frac{3}{2}})^-)$. A direct calculation shows that in this case 
$$
\bV(s)=\begin{bmatrix}
  -\frac{\big(-1+\zeta_3^2\cdot f(s)\big)\cdot\big(-1+f(s)\big)}{3\cdot(\zeta_3+1)\cdot f(s)^2} & \frac{(-1+\zeta_3\cdot f(s))\cdot\big(-1+f(s)\big)}{3\cdot \zeta_3\cdot f(s)^2} & \frac{\big(-1+\zeta_3\cdot f(s)\big)\cdot \big(-1+\zeta_3^2\cdot f(s)\big)}{3\cdot f(s)^2}\\  
   \frac{\zeta_3^2\cdot f(s)+f(s)-2}{3\cdot (\zeta_3+1)\cdot f(s)^2} & -\frac{\zeta_3\cdot f(s)+f(s)-2}{3\cdot \zeta_3\cdot f(s)^2} & -\frac{\zeta_3^2\cdot f(s)+\zeta_3\cdot f(s)-2}{3\cdot f(s)^2}\\
   -\frac{1}{3\cdot (\zeta_3+1)\cdot f(s)^2}  & \frac{1}{3\cdot \zeta_3\cdot f(s)^2} & \frac{1}{3\cdot f(s)^2} 
 \end{bmatrix},
$$
and by Theorem \ref{thm: opt basis trivial} one can take for an optimal basis of $(T)_\vphi,D_s)$ at $0$ to be the set:
\begin{align*}
&\Bigg\{\bV(s)\,\begin{bmatrix}
         1\\
         1\\
         1\\
        \end{bmatrix},\bV(s)\,\begin{bmatrix}
         0\\
         1\\
         0\\
        \end{bmatrix},\bV(s)\,\begin{bmatrix}
         0\\
         0\\
         1\\
        \end{bmatrix}\Bigg\}=\\
        &\Bigg\{\begin{bmatrix}
                                      1\\
                                      0\\
                                      0
                                     \end{bmatrix}
,\begin{bmatrix}
\frac{(-1+\zeta_3\cdot f(s))\cdot\big(-1+f(s)\big)}{3\cdot \zeta_3\cdot f(s)^2} \\
-\frac{\zeta_3\cdot f(s)+f(s)-2}{3\cdot \zeta_3\cdot f(s)^2}\\
\frac{1}{3\cdot \zeta_3\cdot f(s)^2}
                                     \end{bmatrix},\begin{bmatrix}
\frac{\big(-1+\zeta_3\cdot f(s)\big)\cdot \big(-1+\zeta_3^2\cdot f(s)\big)}{3\cdot f(s)^2}\\
-\frac{\zeta_3^2\cdot f(s)+\zeta_3\cdot f(s)-2}{3\cdot f(s)^2}\\
\frac{1}{3\cdot f(s)^2} 
                                     \end{bmatrix}\Bigg\},
\end{align*}
where the latter two solutions have radius of convergence equal to $|3|^{\frac{3}{2}}$.

\bibliographystyle{plain}
\bibliography{biblio}
\end{document}